\documentclass[12pt,reqno]{amsart} 
\usepackage{amssymb} 
\usepackage{bm} 
\usepackage{color}
\usepackage{enumerate}
\usepackage{fancyhdr}
\usepackage{graphicx}
\usepackage{mathrsfs}
\usepackage{setspace}
\allowdisplaybreaks[4]
\renewcommand{\subjclassname}{\textup{2010} Mathematics Subject Classification}

  \newlength{\basicwidth}
  \newlength{\shortbasicwidth}
  \newlength{\basicheight} 
\renewcommand{\contentsname}{
{\normalsize{\rm{\textbf{\normalsize{\hspace{0.1cm}Contents\vspace{0.2cm}}}}}}}

\makeatletter

\renewcommand\subsubsection{\@startsection{subsubsection}{3}%
  \z@{.5\linespacing\@plus.7\linespacing}{-.5em}%
  {\normalfont\bfseries}}
\renewcommand\subsection{\@startsection{subsection}{2}%
   \z@{.5\linespacing\@plus.7\linespacing}{-.5em}%
  {\normalfont\bfseries}}
\renewcommand\section{\@startsection{section}{1}%
  \z@{.7\linespacing\@plus\linespacing}{.5\linespacing}%
  {\large\bfseries\centering}}

\renewcommand{\@secnumfont}{\bfseries}

\renewcommand{\tocsection}[3]{%
  \indentlabel{\@ifnotempty{#2}{\bfseries\ignorespaces  #1 #2.\hspace{0.3cm}}}\bfseries#3}

\newcommand\@dotsep{4.5}
\def\@tocline#1#2#3#4#5#6#7{\relax
  \ifnum #1>\c@tocdepth 
  \else
    \par \addpenalty\@secpenalty\addvspace{#2}%
    \begingroup \hyphenpenalty\@M
    \@ifempty{#4}{%
      \@tempdima\csname r@tocindent\number#1\endcsname\relax
    }{%
      \@tempdima#4\relax
    }%
    \parindent\z@ \leftskip#3\relax \advance\leftskip\@tempdima\relax
    \rightskip\@pnumwidth plus1em \parfillskip-\@pnumwidth
    #5\leavevmode\hskip-\@tempdima{#6}\nobreak
    \leaders\hbox{$\m@th\mkern \@dotsep mu\hbox{.}\mkern \@dotsep mu$}\hfill
    \nobreak
    \hbox to\@pnumwidth{\@tocpagenum{#7}}\par
    \nobreak
    \endgroup
  \fi}

\def\l@section{\@tocline{2}{12pt}{0pc}{5pc}{}}

\makeatother

\def\las#1{\label{#1}\hypertarget{#1}{}}
\def\sref#1{\hyperlink{#1}{\ref*{#1}}}
\def\C{{\mathbb C }}
\def\N{{\mathbb N }}

\def\R{{\mathbb R }}

\def\D{{\mathbb D }}
\def\H{{\mathbb H }} 

\newcommand{\fo}[1]{{\mbox{\footnotesize{$#1$}}}}

\newcommand{\tn}[1]{{\mbox{\tiny{$#1$}}}}

\newcommand{\lar}[1]{{\mbox{\large{$#1$}}}}

\newcommand{\lfrac}[2]{{\lar{\frac{#1}{\vphantom{A^{A^a}}#2}}}}

\newtheorem{theorem}{Theorem}[section]
\newtheorem{corollary}{Corollary}[section]
\newtheorem{lemma}{Lemma}[section]

\numberwithin{equation}{section}
\DeclareMathOperator*{\limsp}{\overline{\lim}}
\DeclareMathOperator{\hol}{\rm{Hol}} 

\DeclareMathOperator{\supp}{\rm{supp}}
\DeclareMathOperator{\IM}{\rm{Im}}
\DeclareMathOperator{\RE}{\rm{Re}}
\newcommand{\ms}{{\mathcal{M}}^+ (\R)}
\newcommand{\PL}{\mathcal{P} }

\def\primP{{\mathcal P}_{\!\scriptscriptstyle\int}}
\def\logP{\log{\mathcal P}}
\def\DeltaprimP{\Delta_{ \!\raisebox{-2pt}{{\small $\primP$}}}}
\def\lm{\C\setminus[1,+\infty)}
\def\br{{\mathcal{B}}(\R)}

\def\ds{\displaystyle}
\let\oldtocsection=\tocsection
\let\oldtocsubsection=\tocsubsection
\let\oldtocsubsubsection=\tocsubsubsection
\renewcommand{\tocsection}[2]{\hspace{0em}\oldtocsection{#1}{#2}}
\renewcommand{\tocsubsection}[2]{\hspace{1.8em}\oldtocsubsection{#1}{#2}}
\renewcommand{\tocsubsubsection}[2]{\hspace{4.5em}\oldtocsubsubsection{#1}{#2}}

\usepackage[linktocpage=true]{hyperref}

\title[Universally starlike and Pick functions]{{\large{Universally starlike and Pick functions}}}

\author{Andrew  Bakan}
\address{ Andrew  Bakan, Institute of Mathematics National Academy of Sciences
         of Ukraine, 01601 Kyiv, Ukraine}
\email{andrew@bakan.kiev.ua}

\author{Stephan Ruscheweyh}
\address{Stephan Ruscheweyh, Institut f\"ur Mathematik, Universit\"at
         W\"urzburg, 97074 W\"urzburg, Germany}
\email{ruscheweyh@mathematik.uni-wuerzburg.de}

\author{Luis Salinas}
\address{Luis salinas, CCTVal \& Departamento de Inform\'atica, UTFSM,
         Valpara\'\i{}so, Chile}
\email{luis.salinas@usm.cl}



\thanks{A. Bakan gratefully accepts the hospitality and the financial
support provided by the UTFSM--Universidad T\'ecnica Federico Santa
Mar\'\i a and the Basal Project FB-0821-CCTVal--Centro Cient\'\i fico
Tecnol\'ogico de Valpara\'\i so, and from the German Academic Exchange
Service (DAAD, Grant 57210233).
S. Ruscheweyh and L. Salinas acknowledge support from UTFSM, FONDECYT
Grant 11500810 and from Basal Project FB-0821-CCTVal.}

\begin{document}

\begin{abstract}
Denote by $\mathcal{P}_{\log}$ the set of all non-constant Pick
functions $f$ whose logarithmic derivatives $f^{\, \prime}/f$ also
belong to the Pick class.  Let $\mathcal{U}(\Lambda)$ be the family
of functions $z\cdot f(z)$, where $f \in\mathcal{P}_{\log}$ and $f$
is holomorphic on $\Lambda:=\C\setminus [1, +\infty)$.
Important examples of functions in $\mathcal{U}(\Lambda)$ are the
classical po\-ly\-lo\-ga\-rithms $Li_\alpha(z)$ $:=$
$\sum_{k=1}^{\infty}z^k/k^\alpha$ for $\alpha\geq 0$, see
\cite{b15}(2015).

In this note we prove that every $\varphi \in \mathcal{U}(\Lambda)$ is
universally starlike, i.e., $\varphi$ maps every circular domain in
$\Lambda$ containing the origin one-to-one onto a starlike domain.
Furthermore, we show that every non-constant function $f \in
\mathcal{P}_{\log}$ belongs to the Hardy space $H_p$ on the upper
half-plane for some constant $p=p(f) > 1$, unless $f$ is proportional
to some function $(a-z)^{-\theta}$ with $a \in \R$ and $0 < \theta
\leq 1$. Finally we derive a necessary and sufficient condition on a
real-valued function $v$ for which there exists $f \in
\mathcal{P}_{\log}$ such that $v (x) = \lim_{\varepsilon \downarrow 0}
{\rm{Im}} f (x + i \varepsilon)$ for almost all $x \in \R$.

\end{abstract}

\maketitle
\thispagestyle{empty}

\vspace{-0.35cm}
\begin{center}
{\small\tableofcontents}
\end{center}


\section{Introduction}

\vspace{0.35cm}
\subsection{Universally starlike functions}$\phantom{a}$ \\

\vspace{-0.35cm}
Universally starlike functions have been introduced in \cite{R},
\cite{rusa-MZ-2009} as
analytic functions in the slit domain $\Lambda:=\C\setminus[1,\infty)$
satisfying the following conditions: (i) $f(0)=f'(0)-1=0$, (ii) $f$ is
univalent in $\Lambda$ and maps any circular domain $\C\subset\Lambda$
containing the origin onto a domain starlike with respect to the origin.
In some subsequent papers this notion has been extended to other domains
and  classes of so-called (universally)
prestarlike functions in $\Lambda$ of  order $\alpha<1$ have been introduced. The
universally starlike functions correspond to the case $\alpha=1/2$ and
some interesting general properties and representations of those
functions have been discovered. In particular, it turns out that these
functions are necessarily Pick functions (in $\Lambda$), and can be
represented in terms of specific probability  measures on the interval
$[0,1]$ (see below).

The main goal of this paper is to provide a complete characterization of the
universally starlike functions, relying on the theory of Pick functions.
The paper also contains many other results, which are interesting by
themselves, regarding the Hardy classes $H_p$ of the upper half-plane.

In  2009 St. Ruscheweyh, L. Salinas and T. Sugawa \cite[Cor.1, p.289]{R}
proved that $\Psi$ is universally starlike if and only if there exists
$\mu\in\mathcal{M}^{+}(\R)$ such that $\mu(\R )=1$,
$\supp\mu\subseteq[0,1]$ and
\begin{equation}\label{1}
\frac{\Psi(z)}{z}
= \exp\left( \int_{[0,1]} \log\frac{1}{1-tz}\,d\mu(t)\right)\,,\quad
  z\in\C\setminus[1,+\infty)\,,
\end{equation}
where the principal branch of the logarithm is taken, $\mathcal{M}^{+}(\R)$
denotes  the cone of all non-negative Borel measures on $\R$ which are finite on compact sets, and
 $$\supp\mu:= \left\{ x\in\R \mid \mu\big( (x-\varepsilon,\,x+\varepsilon) \big)
 >0\ \forall\varepsilon>0 \right\} \, .$$

Let $\mathcal{U}(\Lambda)$ be the family of all functions $z\cdot f(z)$,
where $f$ is holomorphic on $\Lambda:=\C\setminus [1, +\infty)$ and
$f$ is a non-constant Pick function whose logarithmic derivative
$f^{\, \prime}/f$ also belong to the Pick class.

One of the main results of the present paper is that every function
$\varphi\in\mathcal{U}(\Lambda)$ is universally starlike,
i.e., $\varphi$ maps every circular domain in
$\Lambda$ containing the origin one-to-one onto a starlike domain.

\vspace{0.15cm}
Let $\hol(D)$ denote the set of all holomorphic functions in $D\subset\C$,
$C(\R)$ the set of all functions continuous on $\R$,
$\br$ the family of all Borel subsets of
$\R$,\
$\H:= \left\{ z\in\C \mid \IM z>0 \right\}$ and
$\D:= \left\{ z\in\C \mid |z| < 1 \right\}$.

For any $0<p<+\infty $ we consider the Hardy spaces on the unit disk
\cite[p.68]{ko} and on the upper half plane  \cite[p. 168]{dur0},
\cite[p. 112]{ko} as well

\vspace{-0.2cm}
\begin{align*}
H_p(\D)
&:= \left\{ f\in\hol(\D) \mid \|f\|_{H_p(\D)}^{\max\{1,p\}}
    := \sup\limits_{0\leq r<1} \int_{-\pi}^{\,\pi}
    \left|f(re^{i\theta})\right|^p\,d\theta <+\infty \right\}\,, \\[0.2cm]
H_p(\H)
&:= \left\{ f\in\hol(\H) \mid f_{\D}(w) :=
    f\left( \frac{i(1-w)}{1+w}\right)\in H_p(\D) \right\}\,, \\[0.2cm]
\mathcal{H}_p(\H)
&:= \left\{ f\in\hol(\H) \mid \sup\limits_{y>0}
    \int_{\R} |f(x+iy)|^p\,dx < +\infty \right\} \subset H_p(\H)\,.
\end{align*}

\vspace{0.25cm} \noindent
We also deal with the real space $L_{\infty}(\R)$ consisting of all Borel
measurable real-valued functions $f$ defined on $\R$  such that
$\|f\|_{L_{\infty}(\R)} := \inf\left\{ a>0 \mid
 m\big( \{ x\in\R \mid |f(x)|> a \} \big) = 0 \right\} < +\infty$,
where it is assumed that $\inf\emptyset:= +\infty$ and $m$ denotes the
Lebesgue measure on the real line.
Recall that two functions in $L_{\infty}(\R)$ are identified
whenever they are equal almost everywhere with respect to
the Lebesgue measure $m$ (in short: $m$-a.e. or just 'almost everywhere').

\vspace{0.15cm}
Let $\Omega\subset\C$ be a domain containing the origin,
$\hol_{1}(\Omega)$ denote the class of functions in $\hol(\Omega)$
normalized by $f(0)=f'(0)-1=0$ and
$\hol_{0}(\Omega)$ be the set of functions $f\in\hol(\Omega)$ with $f(0)=1$.
A domain is referred to as \emph{circular} when it is a disk or
a half plane.
A domain $\Omega$ is called a \emph{starlike} with respect to the origin if
a straight line segment connecting an arbitrary point in it with the origin
is contained in $\Omega$.
A function $\Psi$  is called \emph{universally starlike} (with respect to
the origin) if $\Psi\in\hol_{1}(\lm)$ and every circular domain in $\lm$
containing the origin is mapped by $\Psi$ one-to-one onto a starlike domain.

\vspace{0.35cm}
\subsection{Pick functions}$\phantom{a}$ \\

\vspace{-0.35cm}

Let ${\mathcal{P}}$ denote the set of Pick functions, i.e. the set of
holomorphic functions $\Phi$ on $\H$  with non-negative imaginary part,
extended to $-\H$ by means of $\Phi(z) = \overline{\Phi(\overline{z})}$
for $z\in-\H$.
Thus,
\begin{equation}\label{2aa}
\mathcal{P}
:= \left\{ \Phi\in\hol(\C\setminus\R) \ \left|\quad
   \frac{\IM\Phi(z)}{\IM z}\geq 0\,, \ z\in\C\setminus\R\,;\
   \Phi(z) = \overline{\Phi(\overline{z})}\,, \ z\in-\H \ \right.\right\}\,.
\end{equation}
It is known (see \cite[p.31]{berg}) that every $\Phi\in \PL$ is
either a real constant or $\IM\Phi(z)> 0$ for all $z\in \H$.
Each function  in the class $\mathcal{P}$ has a unique canonical
representation of the form (see \cite[Th.1, p.20]{D}, \cite[Lem.2.1, p.23]{sho})
\begin{equation}\label{2}
\Phi(z)
= \alpha z + \beta +
  \int_{-\infty}^{+\infty} \left( \frac{1}{t-z}-\frac{t}{1+t^2} \right)\,d\sigma(t)\,,
  \quad z\in\C\setminus\R\,,
\end{equation}
where $\alpha\geq 0$, $\beta\in\R$ and $\sigma\in\mathcal{M}^{+}(\R)$
for which $\int_{\R}(1+t^2)^{-1}\,d\sigma(t)$ is finite.
Conversely, any function of the form \eqref{2} is in $\mathcal{P}$.
It has been proved in \cite[Lem.2, p.26]{D} that
\begin{equation}\label{a2}
\Phi\in\mathcal{P} \cap \hol\big( \,(a,b)\, \big)
\quad\Leftrightarrow\quad
\supp\sigma\subset\R\setminus(a,b)\,,
\end{equation}
where $a,b\in\R$, $a<b$ and $\sigma$ is a measure corresponding to $\Phi$
by \eqref{2}.


For arbitrary $\Phi\in\mathcal{P}\setminus \{0\}$  we have
$\log\Phi\in\mathcal{P}$ and any function
\begin{equation*}
f\in\log\mathcal{P}
:= \big\{ \log\Phi \mid \Phi\in\mathcal{P}\setminus\{0\} \big\}
\subset\mathcal{P}
\end{equation*}
can be uniquely represented in the form (see \cite[p.27]{D})
\begin{equation}\label{3}
f(z) = \beta + \int_{-\infty}^{+\infty}
\left( \frac{1}{t-z}-\frac{t}{1+t^2} \right)\,\rho(t)\,d t\,,\quad
z\in\C\setminus\R\,,
\end{equation}
where $\beta\in\R$ and $\rho\in L_{\infty}(\R)$ satisfies
$0\leq\rho(x)\leq 1$ for almost all $x\in\R$ with respect to
the Lebesgue measure (in short: {\em for almost all\/}).
Conversely, any function of the form \eqref{3} is in $ \log\mathcal{P}$.

\vspace{0.25cm}
\subsection{Boundary values of Pick functions}\las{sec1.2BVPickFus}$\phantom{a}$ \\

\vspace{-0.35cm}
For arbitrary  $\Phi\in\mathcal{P}\setminus\{0\}$ we denote
$f:=\log\Phi\in\log\mathcal{P}$ and replace the upper half plane by the
unit disk with the help of the change of variables
\begin{equation*}
\Phi_{\D}(w) :=
\Phi\left( \frac{i(1-w)}{1+w} \right)\,,\quad
f_{\D}(w) = \log\Phi_{\D}(w)\,,\quad w\in\D\,.
\end{equation*}
Then the formula \eqref{3} can be rewritten as follows (cp.\cite[p.76]{ko})
\begin{equation}\label{30}
f_{\D}(w)
= \beta + \frac{i}{2}\int_{-\pi}^{\pi} \frac{\,e^{i\theta}+w\,}{e^{i\theta}-w}\
  \rho\left( \tan\frac{\theta}{2} \right)\; d\theta\,.
\end{equation}
In view of A.E. Plessner's theorem \cite[p.17]{ples}
(see \cite[Ex.24, p.245]{tit1}, \cite[p.58]{ko}),
$f_{\D}\in\hol(\D)$ and $\IM f_{\D}\geq0$ in $\D$ imply that for almost all
$\theta\in[-\pi,\pi]$ there exists a finite limit $f_{\D} (e^{i \theta})$ of
$f_{\D}(w)$ as $w\in\D$ tends to $e^{i \theta}$ non-tangentially
(in short: $w \overset{\angle}{\longrightarrow} e^{i \theta} $)
\cite[p.11, 15]{ko}.
Thus,  for almost all $x\in\R$ there exists the (non-tangential) finite
boundary value
$f (x+i0) := \lim\limits_{{\H}\ni z\overset{\angle}{\rightarrow} x } f(z)$ and since
$f( \tan\frac{\theta}{2}+i\,0 ) = f_{\D}(e^{i \theta})$, $\theta\in(-\pi,\pi)$,
in view of Fatou's \cite[p.11]{ko} and Titchmarsh's \cite[p.20]{ko} theorems
we have
\begin{equation}\label{30a}
f\left( \tan\lfrac{\theta}{2}+i\,0 \right)
=  i\pi \rho\left( \tan\lfrac{\theta}{2} \right) + \beta +
   \lim_{0<\varepsilon\to0} \int_{\varepsilon}^{\pi}
   \frac{ \rho\left( \tan\lfrac{(\theta+t)}{2} \right) -
          \rho\left( \tan\lfrac{(\theta-t)}{2} \right)}
        { 2\tan\lfrac{t}{2} }\;dt
\end{equation}
for almost all $\theta\in(-\pi,\pi)$.
For the corresponding $x \in\R$ the limit
\begin{equation*}
\Phi(x+i0):= \lim_{{\H} \ni z \overset{\angle}{\rightarrow} x} \Phi(z)
= \lim_{\H\ni z\overset{\angle}{\rightarrow} x} e^{f(z)}
= e^{f(x+i0)}
\end{equation*}
exists and is finite.
Moreover, as straightforward consequences of well-known results in the
theory of $H_p$-spaces,  it is established in subsection~\ref{31} below
that (cp. \cite[Lem. 6.2, p. 255]{ost})
\begin{alignat}{2}
& \int\nolimits_{\R}
  \frac{|\Phi(x+i0)|^{{\fo{\pm\delta}}}}{1+x^2}\,dx  < +\infty\,,
& \quad \delta\in(0,1)\,,
& \quad \Phi\in\PL\setminus\{0\}\,, \label{3a} \\
&\int\nolimits_{\R}
  \frac{|f(x+i0)|^{{\fo{p}}}+e^{\, {\fo{\delta|f(x+i0)|}}}}{1+x^2}\;dx < +\infty\,,
& \quad \delta\in(0,1)\,,
& \quad p > 0\,,\quad f\in\log{\mathcal{P}}\,, \label{3b} \\
& \PL \subset H_\delta(\H)\,,\quad
  \log{\mathcal{P}} \subset H_p(\H)\,,
& \quad \delta\in(0,1)\,,
& \quad p > 0\,.\label{3c}
\end{alignat}
It should be noted that  inequality \eqref{3b} cannot be improved to cover
the case $\delta=1$ because $\Phi(z)= z\in\PL\setminus\{0\}$ and
$f(z)=\log z\in\log{\mathcal{P}}$.

Every $\Phi\in\PL\setminus\{0\}$ is an outer function for the class
$H_\delta(\H)$ for arbitrary $\delta\in(0,1)$
(see \cite[Th.11.6, p.193]{dur0}),
as it follows from \eqref{3a}, \eqref{3c} and the Schwarz integral
formula \cite[p.206]{rem} applied to the function
$\log\Phi\in H_1(\H)$, which by \cite[Th.3.1, p.34]{dur0}
can be represented as the Poisson integral
\begin{equation*}
\log\Phi(x+iy) = \frac{1}{\pi} \int_{\R}
\frac{\;y\,\log\Phi(t+i0)\;}{(x-t)^{2} +y^{2}}\;dt\,,\quad
x+iy\in\H\,,\quad\Phi\in\PL\setminus\{0\}\,,
\end{equation*}
of its boundary values $\log\Phi(x+i 0)$.
Thus, in addition to the exponential representation $\Phi=e^{f}$ with
$f=\log\Phi$ given by \eqref{3} with $\pi\rho(t) = \arg\Phi(t+i 0)$,
any $\Phi\in\PL\setminus\{0\}$ has also another one, namely

\vspace{-0.3cm}
\begin{equation*}
\Phi(z)
= \exp\left( \lfrac{1}{\pi i} \int\limits_{-\infty}^{+\infty}
  \left( \lfrac{1}{t-z}-\lfrac{t}{1+t^2} \right)\,
  \log|\Phi(t+i0)|\;dt + i\arg\Phi(i) \right)\,,\quad
  z\in\H\,,\quad \Phi\in\PL\setminus\{0\}\,.
\end{equation*}

\vspace{-0.5cm}
\noindent
It is well-known (see \cite[p.100]{ko}, \cite[Cor.2.6, p.114]{gar})
that if $g\in C(\R)$ is periodic with period $2\pi$, then
\begin{equation}\label{3c0}
\int_{-\pi}^{\,\pi}
e^{\ {\fo{p\cdot|\widetilde{g}(\theta)|}}}\; d\theta < +\infty \quad\text{for all}\ p>0\,,
\end{equation}
where (cp. \cite[p.89]{ko})
\begin{equation}\label{3c1}
\widetilde{g}(\theta)
:= \lim\limits_{0<\varepsilon\to0} \frac{1}{\pi} \int_{\varepsilon }^{\pi}
   \frac{g(\theta+t)-g(\theta-t)}{2\tan\lfrac{t}{2}}\;dt\,.
\end{equation}
Therefore, \eqref{30} and \eqref{30a} yield (compare with \eqref{3b})
\begin{equation}\label{3d}
\int_{\R} \frac{e^{\ {\fo{ q\cdot |f(x+i0)|}}}}{1+x^{2}}\;dx < +\infty\,,\quad q>0\,,
\end{equation}
provided that $f\in\log\mathcal{P}$ and the function $\rho$ corresponding
to $f$ in \eqref{3} satisfies $\rho(\tan\lfrac{x}{2})\in C(\R)$.


\vspace{0.25cm}
\subsection{Special subclasses of ${\mathcal{P}}$}$\phantom{a}$ \\

\vspace{-0.35cm}

It is easily verified (see~\sref{32}) that the power of exponent in the
righthand side of \eqref{1} can be converted to the following
expression
\begin{align}
L_{\mu}(z)
&:= \int_{[0,1]} \log\frac{1}{1-tz}\;d\mu(t) \label{5} \\
& = - \int_{[0,1]} \log\sqrt{1+t^{2}}\,d\mu(t)
    + \int_{1}^{\infty} \left( \frac{1}{t-z} - \frac{t}{1+t^2} \right)\,
      \left(1-\mu(1/t)\right)\;dt\,, \nonumber
\end{align}
where $\mu(x):=\mu\big((-\infty, x)\big)$, $x\in\R$.
Due to \eqref{3}, this means that $L_{\mu}(z)\in\log{\mathcal{P}}$
and hence $\Psi(z)/z\in\PL$ for every universally starlike
function $\Psi$.
Moreover, since $1-\mu(1/t)$ in the righthand side of \eqref{5}
is non-decreasing on $[1,+\infty)$, \eqref{5} singles out, from
the class $\log{\mathcal{P}}$, those functions whose representation
\eqref{3} has only non-decreasing functions $\rho$.
Therefore, the functions $L_{\mu}(z)$ and $\Psi(z)/z$ belong to narrower
classes than $\log\mathcal{P}$ and $\PL$, respectively.
This circumstance leads to the following definitions.

Denote the set of all primitives of the functions in $\mathcal{P}$
as follows
\begin{equation*}
\primP := \ \left\{ f\in\hol\left( \C\setminus\R \right)
\mid f^{\,\prime}(z)\in\mathcal{P}\,,\; f(-i)=\overline{f(i)}\,\right\}\,,
\end{equation*}
and introduce the three following subclasses of ${\mathcal{P}}$,
\begin{gather}
\primP\cap\mathcal{P}
= \big\{ f\in{\mathcal{P}} \mid f^{\,\prime}\in\mathcal{P} \big\}\,, \label{7} \\
\primP\cap\logP = \left\{ f\in\mathcal{P} \mid
e^f\in\mathcal{P}\,,\;f^{\,\prime}\in\mathcal{P} \right\}\,, \label{8} \\
\mathcal{P}_{\log}
:= \left\{ e^f \mid f\in\primP\cap\logP \right\}
=  \left\{ f\in\mathcal{P} \mid \big( \log f \big)^{\,\prime} =
   \frac{f^{\,\prime}}{f}\in\mathcal{P} \right\}\,. \label{9}
\end{gather}
where for the class \eqref{9} the following notations is also used
(cp. \cite[p.18]{D})
\begin{equation*}
\mathcal{P}_{\log}(A):=
\mathcal{P}_{\log}\cap\hol\left( \C\setminus A \right)\,,\quad A\subset\R\,.
\end{equation*}
Observe that each function in any of the classes introduced above is defined
uniquely by its values on the upper half plane, because its values
on the lower half plane have been assigned under a single formula
$f(z)=\overline{f(\overline{z})}$, $z\in-\H$.

In the sequel we  obtain  canonical representations for functions
in the classes $\mathcal{P}\cap\primP$, $\primP\cap\logP$,
$\mathcal{P}_{\log}$ and prove that
$\Psi$ is universally starlike if and only if
\begin{equation*}
\frac{\Psi(z)}{z}\in\mathcal{P}_{\log}\cap\hol_{\hspace{0.025cm}0}(\lm)\,.
\end{equation*}
In Theorem~\ref{th3} below it is proved that for every
$f\in\left( \primP\cap\logP \right)\setminus\DeltaprimP$, where the class
$\DeltaprimP$ of elementary functions is defined in \eqref{2.3},
there exists at least one value of $p>1$ such that
\begin{equation}\label{4ad}
\int_{\R} \exp\big( p\cdot\RE f(x+i0) \big)\;dx < +\infty\,,\quad
p=p(f)>1\,,\quad f\in\left( \primP\cap\logP \right)\setminus\DeltaprimP\,.
\end{equation}
In other words, as it is established in Theorem~\sref{cor3} below,
\begin{equation*}
\mathcal{P}_{\log}\setminus\exp(\DeltaprimP)
\subset\bigcup_{p>1}\mathcal{H}_p(\H)
\end{equation*}
and therefore, every universally starlike function divided by $z$ belongs
to $\cup_{p > 1}\,\mathcal{H}_p (\H)$ provided that it does not belong to
the set $\exp(\DeltaprimP)$ of elementary functions defined in \eqref{2.4}.

To clarify the relationship between inequalities \eqref{4ad}, \eqref{3b} and
\eqref{3d}, observe that
\begin{equation}\label{f1}
\int_{\R} e^{\,{\fo{\pm\delta\cdot\RE f(x+i0)}}}\,dx = +\infty\,,\quad
\delta\in(0,1)\,,\quad f\in\log \mathcal{P}\,,
\end{equation}
(the proof of \eqref{f1} is given in subsection~\sref{31} below).
Furthermore, if
$f\in\left( \primP\cap\log\mathcal{P} \right)\setminus\DeltaprimP$
and $p(f)>1$ is chosen in \eqref{4ad}, then in view of \eqref{th3.3}
\begin{equation*}
e^{\ {\fo{-p\cdot\RE f(x+i0)}}}\notin L_1 \left( \R,\lfrac{dx}{1+x^2}\right)\,,\quad
p\geq p(f)>1\,,\quad
f\in\left( \primP\cap\log\mathcal{P} \right) \setminus \DeltaprimP\,,
\end{equation*}
and the inequality \eqref{4ad} is incompatible with \eqref{3d} for
$q\geq p(f)>1$.
The latter fact follows also from Corollary~\sref{cor1} and \eqref{2.3}
because they show that all functions $\rho$ corresponding by \eqref{3} to
$f\in\left(\primP\cap\log\mathcal{P}\right)\setminus\DeltaprimP$ should be
non-constant and non-decreasing on the whole $\R$ while the property
$\rho\left(\tan\lfrac{x}{2}\right)\in C(\R)$ guaranteeing the validity of
\eqref{3d} means in particular the existence of two finite and equal limits
$\lim_{t \to +\infty} \rho (t) = \lim_{t \to -\infty} \rho (t)$.

\vspace{0.45cm}
\section{Main Results}

\vspace{0.25cm}
In order to formulate the main theorems it is first necessary to introduce
some  notations for the  Lebesgue-Stieltjes measures and integrals.

\vspace{0.25cm}
\subsection[Notation]
{Notation for measures and non-decreasing functions on the real line} \label{21}$\phantom{a}$ \\

\vspace{-0.35cm}

Denote by $M^{\uparrow}(E)$ the family of all non-decreasing on $E \subset \R$
functions $\phi : E \to \R$.
Consider an arbitrary $\phi\in M^{\uparrow}(\R)$ and the Lebesgue-Stieltjes
measure $\omega_{\phi}\in\mathcal{M}^{+}(\R)$ induced by $\phi\,$
\cite[Def. 3.9, p.147]{mp}.
It is worth reminding that at every $x\in\R$ the left- and the right-hand side
limits exist and both are finite:
\begin{equation*}
-\infty < \phi(x-0):= \lim_{t\uparrow x}\phi(t) \leq \phi(x)
\leq \lim_{t\downarrow x}\phi(t)=:\phi (x+0) < +\infty.
\end{equation*}

Moreover, the set of points of discontinuity of $\phi$,
\begin{equation}\label{aux0}
D_{\phi} = \big\{ x\in\R \mid \phi(x-0)<\phi(x+0) \big\}\,,
\quad \phi\in M^{\uparrow}(\R)\,,
\end{equation}
is at most countable \cite[Th.8.19, p.111]{hew}.

The mapping $\mathcal{L_S}$ of $M^{\uparrow }(\R)$ onto $\mathcal{M}^{+} (\R)\,$,
$\phi\mapsto\omega_{\phi}$ (cf. \cite[Th.1.16, p.35]{fo}),
determines a partition of $M^{\uparrow}(\R)$ into disjoint sets
\begin{equation*}
\left\{ \mathcal{L_S}^{-1}(\mu) \mid \mu\in\mathcal{M}^{+}(\R) \right\}
= \left\{ \Gamma_\phi \mid \phi\in M^{\uparrow}(\R) \right\}\,,
\end{equation*}
where
\begin{equation*}
\Gamma_\phi
= \big\{ b+\psi \mid b\in\R\,,\ \psi\in\aleph_\phi \big\}\,,
\end{equation*}
and
\begin{equation}\label{aux10}
\aleph_\phi
= \big\{ \psi\in M^{\uparrow}(\R) \mid
  \psi(x)\in[\phi(x-0),\,\phi(x+0)]\,,\; x\in\R \big\}\,,\qquad
  \phi\in M^{\uparrow}(\R)\,,
\end{equation}
(see \cite[Prop. 3.9(v), p.145]{mp}, \cite[p. 551]{ho}).
We now assume that the closure of $E\subseteq\R$ in $\R$ is equal to $\R$
and $f\in M^{\uparrow}(\R)$.
Then it is  easily  verified that
\begin{equation}\label{aux12}
f(x-0) = \lim_{E\ni\,t\uparrow x} f(t)\,,\qquad
f(x+0) = \lim_{E\ni\,t\downarrow x} f(t)\,,\qquad
x\in\R\,,
\end{equation}
from which we get
\begin{alignat}{2}
\aleph_{\phi}
&=& \big\{ \psi\in M^{\uparrow}(\R) \mid \psi=\phi\
    \text{$m$-a.e. on $\R$}\; \big\}\,, \quad
    & \phi\in M^{\uparrow}(\R)\,, \label{aux4} \\
D_{\psi}
&=& D_{\phi}\,,\quad
    \psi(-\infty)=\phi(-\infty)\,,\quad
    \psi\in\aleph_{\phi}\,,\quad
    &\phi\in M^{\uparrow}(\R)\,, \label{aux11}
\end{alignat}
irrespective of whether $\phi (-\infty):=\lim_{\, t \to -\infty}\phi(t)$
is finite or equal to $-\infty$.
It follows from \eqref{aux10} that for each $\phi\in M^{\uparrow} (\R)$ the
functions in the class $ \aleph_{\,\phi}$ are  equal to each other everywhere except
at most at countable many points in $D_{\phi}$ and (cf. \cite[Ex.2, p.156]{bar})
$ \aleph_{\phi}$ contains a unique function $\psi$  satisfying
\begin{equation}\label{aux2}
\psi(x) = \frac{\psi(x+0) + \psi(x-0)}{2}\,,\quad x\in\R\,.
\end{equation}
If $[a,b]\subset\R$ is a bounded interval and $g$ is an absolutely continuous
function on  $[a,b]$ then the formula of integration by parts
(cf. \cite[p.166]{bar}, \cite[Th.6.2.2,  p.100]{cart},
\cite[Prop.4.9, p.202; Cor.4.3, p.214]{mp})
gives
\begin{equation}\label{5ad}
\int_{[a,b]} g(x)\,d\omega_{\phi}(x) = g(b)\,\phi(b+0) - g(a)\,\phi(a-0)
- \int_{a}^{b} \phi(x)\,g'(x)\,dx\,.
\end{equation}
In the sequel, for any given $\phi\in M^{\uparrow }(\R)$ we  denote the
integral $\int g(x)\,d\omega_{\phi}(x)$ of a function $g$, with respect to
the measure $\omega_{\phi}$ by $\int g(x)\,d\phi(x)$ (see \cite[p.107]{fo}),
and observe  that the value of $\int g(x)\,d\phi(x)$ does not change when
$\phi$ runs over the class $\Gamma_{{\phi}}$.
Furthermore, denote by the same letter $\phi$ the Lebesgue-Stieltjes measure
$\omega_{\phi}$ induced by $\phi\in M^{\uparrow}(\R)$, and as usual, we write
$\phi(A)$, $A\subset\R$, when $\phi$ is understood as a measure and
$\phi(x)$, $x\in\R\cup\{\pm\infty\}$, when $\phi$ is understood as a function.
Under this convention, if $\phi$ is non-negative, then obviously
$\phi(x-0)=\phi(-\infty)+\phi ((-\infty,x))$ for every $x\in\R $.
Each point in $\supp\phi $ is called a {\emph{growing point}} of
$\phi\in M^{\uparrow }(\R)$.

When $\mu\in\mathcal{M}^{+}(\R)$ is given, we also denote by $\mu$ any of
those functions in $\mathcal{L_S^{{\rm{-1}}}} (\mu)\subset M^{\uparrow }(\R)$
which satisfy \eqref{aux2}.
But, if $\mu$ is finite, we additionally assume that $\mu(-\infty) = 0$,
i.e.
\begin{equation*}
2\mu(x) := \mu\big(\,(-\infty,\,x)\big) +
\lim_{\varepsilon\downarrow 0}\mu\big(\,(-\infty,\,x+\varepsilon)\,\big)\,,
\quad x\in\R\,.
\end{equation*}

\vspace{0.25cm}
\subsection{\texorpdfstring{The classes $\primP\cap\mathcal{P}$ and
$\primP\cap\logP$}{The classes (1.15) and (1.16)}} \label{22}$\phantom{a}$ \\

We recall that any harmonic function in $\H$ is analytic in $\H$
(see \cite[Th.1.15, p.31]{hk}) and, in particular, all its partial
derivatives of any order are continuous functions in $\H$.
The following assertion is of independent interest.

\begin{theorem}\label{th1}
Let $V(x,y)$ be a harmonic function in ${\H}$ and
\begin{equation}\label{th1.1}
V(x,y)\geq0\,,\quad V_{x}(x,y)\geq0\,,\quad y>0\,,\quad x\in\R\,.
\end{equation}
Then the following statements hold:

{\rm{\bf{1}}}.
There exist a finite constant $\alpha\geq0$ and a non-negative non-decreasing
function $\nu$ on $\R$ such that
\begin{equation}\label{th1.2}
\int_{1}^{+\infty}\frac{\nu(t)}{t^2}\,dt<+\infty\,,
\end{equation}
and
\begin{equation}\label{th1.3}
V(x,y)
= \alpha y + \int_{-\infty}^{+\infty} \frac{y}{y^2+(x-t)^2}\;\nu(t)\,dt\,,
\quad y>0\,,\quad x\in\R\,.
\end{equation}
Conversely, any function of this form is harmonic in ${\H}$ and satisfies
conditions \eqref{th1.1}.
The representation \eqref{th1.3} is  unique if we identify non-decreasing
functions that are equal almost everywhere on $\R$
{\rm (see (\ref{aux4}), (\ref{aux10}))}.

\smallskip
{\rm{\bf{2}}}. Either
$V_{x}(x,y)>0$ or $V_{x}(x,y)\equiv 0$ for any $y>0$, $x\in\R$,
and the latter case holds if and only if the function $\nu$ in
\eqref{th1.3} is a non-negative constant function $\nu(x)\equiv a\geq0$,
or, what is the same, $V(x,y)=\alpha y + a\cdot\pi$, $y>0$, $x\in\R$,
with some  $\alpha,a\geq0$.
\end{theorem}

\vspace{0.25cm}
Since in all the statements below the subclasses of $\mathcal{P}\setminus\{0\}$
are considered, we introduce the following general notations.

For arbitrary $f \in {\mathcal{P}} \setminus \{0\}$ denote
\begin{equation*}
U^{f}(x,y) := \RE f(x+iy)\,,\quad
V^{f}(x,y) := \IM f(x+iy)\,,\quad x\in\R\,,\quad y>0\,.
\end{equation*}
Let $\mathcal{D}(U^{f})$ ($\mathcal{D}(V^{f})$) be the set of all those
$x\in\R$ where there exists a finite limit $U^{f}(x)$ ($V^{f} (x)$) of
$\RE f(z)$  ($\IM f(z)$) as $z\in\H$ tends to $x$ non-tangentially.
As mentioned in subsection \sref{sec1.2BVPickFus} 
\begin{equation}\label{aux21}
m\left( \R\setminus\mathcal{D}(f) \right)
=  m\left( \R\setminus\mathcal{D}(U^{f}) \right)
=  m\left( \R\setminus\mathcal{D}(V^{f}) \right) = 0\,,\quad
\mathcal{D}(f) := \mathcal{D}(U^{f})\cap\mathcal{D}(V^{f})\,,
\end{equation}
and on the set $\mathcal{D}(f)$ one can define the non-tangential
boundary value of $f$ as follows
\begin{equation*}
f(x+i0) = U^{f}(x)+iV^{f}(x)\,,\quad x\in\mathcal{D}(f)\,.
\end{equation*}

Theorem~\ref{th1} gives rise to a canonical representation of the functions
in the class $\mathcal{P}\cap\primP$ which, in turn,
leads to explicit expressions for their boundary values.

\vspace{0.25cm}
\begin{theorem}\label{th2}\
Let $\primP \cap \mathcal{P}$ be defined as in \eqref{7}.

\smallskip
{\rm{\textbf{ 1.}}}\
A function $\Phi$ belongs to  the class $\primP\cap\mathcal{P}$
if and only if there exist constants $\alpha\geq0$, $\beta\in\R$ and a
non-negative non-decreasing function $\nu$ on $\R$ satisfying
\begin{equation}\label{th2.2}
\int_{1}^{+\infty} \frac{\nu(t)}{t^2}\,dt < + \infty \ ,
\end{equation}
such that
\begin{equation} \label{th2.1}
\Phi(z) = \alpha z + \beta + \int_{-\infty}^{+\infty}
\left( \frac{1}{t-z} - \frac{t}{1+t^2} \right)\,\nu(t)\,dt\,,\quad
z\in\C\setminus\R\,,
\end{equation}
or, what is the same,
\begin{equation}\label{th2.3}
\Phi(z) = \alpha z + \beta + i\pi\nu(-\infty) +
\int_{-\infty}^{+\infty}\log\frac{\sqrt{1+t^2}}{t-z}\,d\nu(t)\,,\quad z\in\H\,,
\end{equation}
where the principal branch of the logarithm is taken and the integral in
\eqref{th2.3} converges absolutely since
\begin{equation}\label{th2.4}
    \int_{1}^{+\infty}   \frac{d \, \nu (t)}{t}\  < + \infty \ ,
\end{equation}
as  it follows from \eqref{th2.2}.
The representations \eqref{th2.1} and \eqref{th2.3} are unique
if we identify the non-decreasing functions that are equal almost everywhere
on $\R$ {\rm (see \eqref{aux10}, \eqref{aux4}, \eqref{aux11} )}.

\smallskip
{\rm {\textbf{2.}}}\
If $\Phi\in\mathcal{P}\cap\primP$, then for all $y>0$ and $x\in\R$ the
following equalities hold

\vspace{-0.3cm}
\begin{align}
U^{\Phi}(x,y)
&= \alpha x + \beta + \int_{-\infty}^{+\infty}
   \left( \frac{t-x}{y^2+(x-t)^2}-\frac{t}{1+t^2} \right)\,\nu(t)\,dt
   \label{th2.5} \\
&= \alpha x + \beta  +
   \int_{-\infty}^{+\infty} \log\sqrt{\frac{1+t^2}{y^2+(t-x)^2}}\,d\nu(t)\,,
   \nonumber \\[0.3cm]
V^{\Phi} (x,y)
&= \alpha y + \int_{-\infty}^{+\infty} \frac{y}{y^2+(x-t)^2}\,\nu(t)\,dt
   \label{th2.6} \\
&= \alpha y + \pi\nu(-\infty) + \int_{-\infty}^{+\infty}
   \left( \frac{\pi}{2} -\arctan\frac{t-x}{y} \right)\,d\nu(t)\,.
   \nonumber
\end{align}

\vspace{0.1cm} \noindent
For almost all $x\in \R$ we have
$\log\lfrac{\sqrt{1+t^2}}{|t-x|}\in L_1(\R,d\nu(t))$
{\rm (cf. \cite[Lem. 2, p.86]{tw})} and
\begin{equation}\label{th2.7}
U^{\Phi}(x) = \alpha x + \beta  +
\int_{-\infty}^{+\infty} \log \frac{\sqrt{1 + t^2}}{|t-x|} \ d  \nu (t) \ .
\end{equation}
The limit of $V^{\Phi}(x,y)$ as $y\to 0$, $y>0$
{\rm (in short: $y\downarrow 0$)}, exists and is finite for every
$x\in\R$ and
\begin{equation}\label{th2.8}
V^{\Phi}(x,0+0) := \lim_{y \downarrow 0} V^{\Phi}(x,y) = \pi\nu(x)\ ,
\end{equation}
provided that $\nu$ satisfies \eqref{aux2}.
In addition, $\mathcal{D}(V^{\Phi})\subset\R\setminus D_{\nu}$
{\rm (see \eqref{aux0}, \eqref{aux21})}.
\end{theorem}

\vspace{0.25cm}
Some important particular cases of functions from the class
$\mathcal{P}\cap\primP$ are considered in the following corollary.


\vspace{0.25cm}
\begin{corollary}\label{corr1}
The first derivative of any function $\Phi\in\mathcal{P}\cap\primP$
with integral representation \eqref{th2.3} can be expressed as:
\begin{equation}\label{th2.9}
\Phi^{\, \prime}(z)
= \alpha + \int_{-\infty}^{+\infty} \frac{d\nu(t)}{t-z}\,,\quad z\in\H\,.
\end{equation}
and the following assertions hold:
\begin{itemize}
\item[(i)]
If $\Phi$ is a constant function, then $\alpha=0$,
$\nu(x) \equiv \nu(-\infty) \geq 0$ and thus,\\
$\Phi(z) \equiv \beta + i\pi\cdot\nu(-\infty)$,
$\IM\Phi(z) \equiv \pi \cdot \nu(-\infty) \geq 0$,
$\Phi^{\,\prime}(z)\equiv 0$.\vspace{0.2cm}
\item[(ii)]
If $\nu$ is a constant function, i.e. $\nu(x)\equiv\nu(-\infty)\geq 0$,
then $\Phi(z) = \alpha z + \beta  +  i\pi\cdot\nu(-\infty)$ and
$\Phi^{\,\prime} (z) \equiv \alpha$, $\IM\Phi^{\,\prime}(z) \equiv 0$.\vspace{0.2cm}
\item[(iii)]
If $\Phi$ is not a linear function, then $\IM\Phi(z)>0$ and
$\IM\Phi^{\,\prime}(z)>0$ for each $z\in\H$.
\end{itemize}
\end{corollary}


It follows from \cite[Th. 2, p.143]{kats} and \eqref{th2.9} that:
\begin{equation*}
\int_{1}^{+\infty}\frac{\IM\Phi^{\,\prime}(it)}{t}\,dt<+\infty\quad
\text{for every $\Phi\in\mathcal{P}\cap\primP$}\,.
\end{equation*}

A direct application of Theorem~\sref{th2}, with the additional
restriction that $\Phi\in\log{\mathcal{P}}\subset\mathcal{P}$,
yields a canonical representation of all functions in
$\primP\cap\log\mathcal{P}\subset\primP\cap\mathcal{P}$,
with an explicit expressions for their boundary values.


\vspace{0.2cm}
\begin{theorem}\label{cor1}
Let  $\primP\cap\log\mathcal{P}$ be defined as in \eqref{8}.
Every function $f$ in the class $\primP\cap\log\mathcal{P}$
admits a canonical representation of the form
\begin{equation}\label{cor1.2}
f(z) = \beta + \int_{-\infty}^{+\infty}
\left( \frac{1}{t-z} - \frac{t}{1+t^2} \right)\,\nu(t)\,dt\,,\quad
z\in\C\setminus\R\,,
\end{equation}
or, what is the same,
\begin{equation}\label{cor1.3}
f(z) = \beta + i\pi\nu(-\infty) + \int_{-\infty}^{+\infty}
\log\frac{\sqrt{1+t^2}}{t-z}\,d\nu(t)\,,\quad z\in\H\,,
\end{equation}
where the principal branch of the logarithm is taken, $\beta\in\R$
and $\nu$ is a  non-decreasing function on $\R$ satisfying
\begin{equation}\label{cor1.1}
0\leq\nu(x_1)\leq\nu(x_2)\leq1\,,\quad x_1<x_2\,,\quad x_1,x_2\in\R\,.
\end{equation}
Conversely, any function of the form \eqref{cor1.2} or \eqref{cor1.3}
with a non-decreasing function $\nu $ on $\R$ satisfying
\eqref{cor1.1} belongs to $\primP\cap\log\mathcal{P}$.
The representations \eqref{cor1.2} and \eqref{cor1.3} are unique
if we identify the non-decreasing functions that are equal almost
everywhere on $\R$ {\rm (see \eqref{aux10}, \eqref{aux4}, \eqref{aux11})}.

For almost all $x\in\R$ we have
$\log\lfrac{\sqrt{1+t^2}}{|t-x|}\in L_1(\R, d\nu(t))$ and
\begin{equation}\label{cor1.6}
U^{f}(x) = \beta + \int_{-\infty}^{+\infty}
\log\frac{\sqrt{1+t^2}}{|t-x|}\,d\nu(t)\,.
\end{equation}
Whenever $\nu$ satisfies \eqref{aux2}, then
\begin{equation}\label{cor1.7}
V^{f}(x, 0+0) := \lim_{y\downarrow 0} V^{f}(x,y) = \pi\nu(x)\,,\quad
\text{for every $x\in\R$}\,.
\end{equation}
Furthermore, for $f\in\primP\cap\log\mathcal{P}$ the following holds:
\begin{itemize}
\item[(i)]
The function $f$ is identically constant if and only if $\nu$ is a
constant function.\vspace{0.1cm}
\item[(ii)]
If $f$ is a constant function, then $f(z)=\beta+i\pi\theta$ and
$\nu(x)\equiv\theta$, where $\beta\in\R$ and $\theta\in[0,1]$.\vspace{0.1cm}
\item[(iii)]
If $f$ is a non-constant function, then $\IM f(z)>0$ and $\IM f'(z)>0$
for each $z\in\H$.
\end{itemize}
\end{theorem}


\vspace{0.25cm}
It obviously follows from \eqref{cor1.2} and \eqref{cor1.3} that
for arbitrary $f\in\primP\cap\log\mathcal{P}$, $y > 0$
and $x\in\R$ the following equalities hold:
\begin{align*} 
U^{f}(x, y)
&= \beta + \int_{-\infty}^{+\infty}
   \left( \frac{t-x}{y^2+(x-t)^2} - \frac{t}{1+t^2} \right)\,\nu(t)\,dt \\
&= \beta + \int_{-\infty}^{+\infty}
   \log\sqrt{\frac{1+t^2}{y^2+(t-x)^2}}\,d\nu(t)\,, \\
V^{f}(x,y)
&= \int_{-\infty}^{+\infty} \frac{y}{y^2+(x-t)^2}\,\nu(t)\,dt \\
&= \pi\nu(-\infty) + \int_{-\infty}^{+\infty}
   \left( \frac{\pi}{2} - \arctan\frac{t-x}{y} \right)\,d\nu(t)\,.
\end{align*}
The following fact follows easily from Theorem~\sref{cor1}.


\vspace{0.25cm}
\begin{corollary}\label{cor2}
$f\in\primP\cap\log\mathcal{P}\cap\hol ((-\infty,1))$
if and only if there exists some $\mu\in\mathcal{M}^{+}(\R)$
with $\supp\mu\subseteq[0,1]$ and $\mu(\R)=1$, such that
\begin{equation}\label{cor2.1}
f(z) = \int_{[0, 1]} \log\frac{1}{1-tz}\,d\mu(t)\,, \quad z\in\H\,.
\end{equation}
If $f$ is identically constant in \eqref{cor2.1},
then $f\equiv 0$, $\supp\mu = \{0\}$ and $\mu(\{0\})=1$.
\end{corollary}

\vspace{0.1cm}
According to the definition \eqref{8} and Theorem~\sref{cor1}
a non-constant  function $f$ belongs to the class
$\primP\cap\log\mathcal{P}$ {\em if and only if\/}
for every $x+iy\in\H$, $0<V^{f}(x,y)<\pi$ and $V_x^f(x,y)>0$.
But it follows from $V_x^f= -U_y^f$ that the latter inequality is equivalent
to $U_y^f(x,y)<0$, $x+iy\in\H$.
Thus,  a non-constant function $f$ belongs to the class
$\primP\cap\log\mathcal{P}$ {\em if and only if\/} any one
(and therefore all) of the following conditions hold:
\begin{equation}\label{a1}\hspace{-0.5cm}
\left\{
\begin{array}{llll}
(a) & 0<V(x_1,y)<V(x_2,y)< \pi\,,
    & -\infty\!<\!x_1\!<\!x_2\!<\!+\infty\,,
    & y\!>\!0\,, \\[0.3cm]
(b) & U(x,y_2)<U(x,y_1)\,,\ 0<V(x,y)<\pi\,,
    & 0\!<\!y_1\!<\!y_2\!<\!+\infty\,,
    & y\!>\!0\,,\ x\!\in\!\R\,.
\end{array}
\right.\hspace{-0.15cm}
\end{equation}
Indeed, the inequalities in the left-hand side of \eqref{a1}, (a) and (b),
follow obviously from the inequalities $V_x^f(x,y)>0$ and $U_y^f(x,y)<0$,
respectively.
Conversely, \eqref{a1}, (a) and (b), imply that the harmonic functions
$V_x^f(x,y)$ and $-U_y^f(x,y)$ are non-constant and non-negative,
which means that they are strictly positive in $\H$ in view of the mean
value property of harmonic functions \cite[(1.5.4), p.34]{hk}.

An exceptional position in the class $\primP\cap\log\mathcal{P}$
is played only by those functions whose measure $\nu$ given by
\eqref{cor1.3} have a support consisting of at most  one point.

If $\nu$ in \eqref{cor1.2} is a constant function (i.e., $\nu(\R)= 0$
for the corresponding measure $\nu$) then by \eqref{cor1.3}
$f\equiv\beta+i\pi\nu(-\infty)$ is a constant function as well,
where $\nu(-\infty)\in[0,1]$ by \eqref{cor1.1}.
If $\text{supp}\,\nu$ consists of only one point $a\in\R$ and
$\theta:=\nu(\R)\in(0,1]$, then obviously $\nu(-\infty)\in[0,1-\theta]$
and we can introduce a number $\theta_1\in[0,1]$ in order to have
$\nu(-\infty)=\theta_1\cdot(1-\theta)$ and
\begin{equation*}
\nu_{\theta_1,a,\theta}(x) =
\begin{cases}
\theta_1\,(1-\theta)\,, & x< a\,, \\[0.3cm]
\theta_1\,(1-\theta)+{\theta}/{2}\,, & x=a\,, \\[0.3cm]
\theta_1\,(1-\theta)+\theta\,, & x > a\,,
\end{cases}
\qquad\text{where $\theta\in(0,1]$ and $\theta_1\in[0,1]$}\,.
\end{equation*}
If here we put $\theta=0$, we get a constant function
$\nu_{{\fo{\theta_1, a, 0}}}(x)\equiv\theta_1 = \nu(-\infty)$.
By \eqref{cor1.3}, the measure $\nu_{{\fo{\theta_1, a, \theta}}}$ corresponds to the
function of the following form
\begin{equation*}
f_{{\fo{\theta_1, a, \theta}}}(z)
= \beta + i\pi\theta_1\,(1-\theta) +
  \theta\,\log\frac{\sqrt{1+a^2}}{a-z}\,,
\end{equation*}
and so we can define an exceptional class of functions in
$\primP\cap\log\mathcal{P}$ as follows
\begin{align}
\Delta_{\primP}
&= \left\{
   \beta + i\pi\theta_1\,(1-\theta) +
   \theta\cdot\log\frac{\sqrt{1+a^2}}{a-z} \ \, \Big| \ \,
   \theta,\theta_1\in[0,1]\,,\  a,\beta\in\R\,\right\} \label{2.3} \\
&= \left\{
   \beta + i\pi\theta_1\cdot(1-\theta) +
   \theta\cdot\log\frac{1}{a-z} \ \, \Big| \ \,
   \theta,\theta_1\in[0,1]\,,\  a,\beta\in\R\,\right\}\,. \nonumber
\end{align}
Our key result is contained in the following theorem.


\vspace{0.25cm}
\begin{theorem}\label{th3}
Let $\primP\cap\log \mathcal{P}$ and $\Delta_{\primP}$ be defined as in
\eqref{8} and \eqref{2.3}, respectively.
Let furthermore $f\in\primP \cap \log\mathcal{P}$, $\nu$ be a measure
corresponding to $f$ in the integral representation \eqref{cor1.3} and\vspace{-0.2cm}
\begin{equation}\label{th3.1}
J_{\nu}(A):=\sup_{x\in A}\nu\left(\left\{x\right\}\right)\,,\quad
A\subseteq\R\,.
\end{equation}
Then the support of $\nu$ consists of at least two different points
if and only if
$f\in\left(\primP\cap\log\mathcal{P}\right)\setminus\Delta_{\primP}$.
In this case
$J_\nu(\R)<\nu(\R)\leq1$ and for every $0<y_1<y_2<+\infty$ and
${1}/{\nu(\R)}<p<{1}/{J_{\nu}(\R)}$, the following relations
hold
\begin{align}
&\hspace{-0.3cm} \int\limits_{-\infty}^{+\infty}
\exp\left( p\cdot U^f(x,y_2) \right)\,dx
< \int\limits_{-\infty}^{+\infty}
  \exp\left( p\cdot U^f(x,y_1) \right)\,dx
< \int\limits_{-\infty}^{+\infty}
  \exp\left( p\cdot U^f(x) \right)\,dx
< +\infty\,, \label{th3.2} \\[0.3cm] 
&\hspace{-0.3cm}  \exp\left( -p\cdot U^f(x,y)\right)\,,\, \exp\left( -p\cdot U^f(x)\right)
\notin L_1\left(\R,\frac{dx}{1+x^2}\right)\,,\quad y>0\,,\quad
p>\frac{1}{\nu(\R)}\geq1\,. \label{th3.3}
\end{align}
\end{theorem}

\vspace{0.25cm}
\subsection{\texorpdfstring{The class $\mathcal{P}_{\log}$}
           {The class (1.17)}} \label{22b}$\phantom{a}$ \\

\vspace{-0.15cm}
The next characteristic properties of the functions in the class
$\mathcal{P}_{\log}$ follow immediately from \eqref{9} and \eqref{a1}.

\vspace{0.25cm}
\begin{theorem}\label{pr1}
A non-constant function $\varphi\in\mathcal{P}$ belongs to $\mathcal{P}_{\log}$
iff one of the following conditions holds:
\begin{align*}
(1) & \qquad
\big| \varphi(x+iy_1) \big| > \big| \varphi(x+iy_2) \big|\,,
\quad x\in\R\,,\quad 0 < y_1 < y_2 < +\infty\,; \\[0.2cm]
(2) & \qquad
\arg\varphi(x_1+iy) < \arg\varphi(x_2+iy)\,,
\quad y>0\,,\quad -\infty < x_1 < x_2 < +\infty\,; \\[0.2cm]
(3) & \qquad
\begin{vmatrix}
U^{\varphi}(x_1,y) & U^{\varphi}(x_2,y) \\ V^{\varphi}(x_1,y) & V^{\varphi}(x_2,y)
\end{vmatrix}
> 0\,,\quad y>0\,,\quad -\infty < x_1 < x_2 < +\infty\,; \\[0.2cm]
(4) & \qquad
U^\varphi(x,y)\, U_y^\varphi(x,y) +  V^\varphi(x,y)\, V_y^\varphi(x,y)
< 0\,,\quad x\in\R\,,\quad y > 0\,.
\end{align*}
\end{theorem}

\smallskip
According to the definition \eqref{9}, the corresponding exceptional class
of functions in $\mathcal{P}_{\log}$ can be written in the following manner
\begin{equation}\label{2.4}
\exp(\Delta_{\primP})
= \left\{
  \frac{ \exp\left( \beta+i\pi\theta_1\,(1-\theta) \right)}{(a-z)^{\theta}}
  \ \Big| \ \theta,\theta_1\in[0,1]\,,\quad a,\beta\in\R\,\right\}\,.
\end{equation}
In the sequel, for arbitrary function $v\in\bigcup_{p>1} L_p(\R,dx)$
we  use of the notation
\begin{equation}\label{2.5}
\widetilde{v}(x):=
\frac{1}{\pi} \cdot \lim_{\varepsilon\downarrow 0}
\int_{\varepsilon}^{+\infty} \frac{v(x+t)-v(x-t)}{t}\; dt
= \frac{1}{\pi} \cdot P\int_{-\infty}^{+\infty} \frac{v(t)\;dt}{t-x}\;,
\end{equation}
where $\ds P\int_{\R}\frac{v(t)}{t-x}\,dt$ denotes the  principal value
of the integral, defined as
$\ds\lim_{\varepsilon\downarrow 0}\int_{|t-x|>\varepsilon}
\frac{v(t)}{t-x}\,dt\,$.
The following assertion is a simple consequence of Theorem~\sref{th3}
and Theorem~\sref{cor1}.

\vspace{0.25cm}

\vspace{0.25cm}
\begin{theorem}\label{cor3}
Let $\mathcal{P}_{\log}$, $\exp(\Delta_{\primP})$ and ${J}_{\nu}(\R)$ be
defined as in \eqref{9}, \eqref{2.4} and  \eqref{th3.1}, respectively.
Then:

\smallskip
{\rm {\textbf{1.}}}
Every function $\varphi$ in the class $\mathcal{P}_{\log}$ admits an
exponential representations of the form

\vspace{-0.2cm}
\begin{align}
\varphi(z)
&= \exp\left[ \beta + \int_{-\infty}^{+\infty}
   \left( \frac{1}{t-z} - \frac{t}{1+t^2} \right)\,
   \nu(t)\,dt\, \right] \label{cor3.1} \\[0.25cm]
&= \exp\left[ \beta + i\pi\nu(-\infty) + \int_{-\infty}^{+\infty}
   \log\frac{\sqrt{1+t^2}}{t-z}\,d\nu(t)\,\right]\,, \quad
   z\in\H\,, \nonumber
\end{align}

\vspace{0.2cm} \noindent
where the principal branch of the logarithm is taken, $\beta\in\R$ and
$\nu$ is a non-decreasing function on $\R$ satisfying
$0\leq\nu(x_1)\leq\nu(x_2)\leq 1$ and $-\infty<x_1<x_2<+\infty$.

Conversely, any function of one of the forms \eqref{cor3.1} with
$\nu\in M^{\uparrow}(\R)$ satisfying \eqref{cor1.1} belongs to
$\mathcal{P}_{\log}$.

The representations \eqref{cor3.1} are unique if we identify
non-decreasing functions that are equal almost everywhere on
$\R$ {\rm{(}}see \eqref{aux4}, \eqref{aux10}, \eqref{aux11}{\rm{)}}.

For almost all $x\in\R$ we have
$\log\lfrac{\sqrt{1+t^2}}{|t-x|}\in L_1(\R,d\nu(t))$ and
\begin{equation}\label{cor3.6}
\varphi(x+i0)
= \exp\left( \beta + \int_{-\infty}^{+\infty}
  \log\lfrac{\sqrt{1+t^2}}{|x-t|}\,d\nu(t) + i\pi\nu(x) \right)\,.
\end{equation}

\smallskip
{\rm {\textbf{2.}}}
Let $\varphi\in\mathcal{P}_{\log}\setminus\exp(\Delta_{\primP})$ and $\nu$
be a measure corresponding to $\varphi$ in the righthand side of the
exponential representation in \eqref{cor3.1}.
Then $J_{\nu}(\R) < \nu(\R) \leq 1$ and for every
${1}/{\nu(\R)}<p<{1}/{J_{\nu}(\R)}$ and $0<y_1<y_2<+\infty$
the following inequalities hold

\vspace{-0.2cm}
\begin{equation*}
\int_{-\infty}^{+\infty} |\varphi(x+iy_2)|^p\,dx
< \int_{-\infty}^{+\infty} |\varphi(x+iy_1)|^p\,dx
< \int_{-\infty}^{+\infty} |\varphi(x+i0)|^p\,dx
< +\infty
\end{equation*}
and
\begin{equation}\label{cor3.8}
\varphi(x+i0)
= \widetilde{v}(x) + iv(x) \quad\text{\rm $m$-a.e. on $\R$}\,,
\end{equation}
where $\widetilde{v}(x),v(x)\in L_{p}(\R,dx)$, $v(x):= V^{\varphi}(x)$
and $\widetilde{v}$ is defined by formula \eqref{2.5}.
In particular,

\vspace{-0.2cm}
\begin{equation*}
\mathcal{P}_{\log}\setminus\exp(\Delta_{\primP})\subset
\bigcup_{p > 1} \mathcal{H}_p(\H)\,.
\end{equation*}
\end{theorem}

We observe, that the righthand sides of \eqref{cor3.1} and \eqref{cor3.6}
can be written as 
\begin{align*}
U^{\varphi}(x,y)\!
&= \!\!\left[ \cos\left(\! \pi\nu(-\infty) + \!\! \int\limits_{-\infty}^{+\infty}
   \left( \scalebox{1.0}{$\frac{\pi}{2}-\arctan\frac{t-x}{y}$} \right)\,
   d\nu(t) \!\right)\right]\exp\left(\! \beta + \!\! \int\limits_{-\infty}^{+\infty}
   \log\sqrt{ \scalebox{0.8}{$\ds\frac{1+t^2}{y^2+(t-x)^2}$} }\,d\nu(t) \!\right)
  \!, \\
V^{\varphi}(x,y)\!
&=\!\!\left[ \sin\left(\! \pi\nu(-\infty) +\!\! \int\limits_{-\infty}^{+\infty}
   \left( \scalebox{1.0}{$\frac{\pi}{2}-\arctan\frac{t-x}{y}$} \right)\,
   d\nu(t) \!\right)\right] \exp\left(\! \beta + \!\! \int\limits_{-\infty}^{+\infty}
   \log\sqrt{ \scalebox{0.8}{$\ds\frac{1+t^2}{y^2+(t-x)^2}$}}\,d\nu(t) \!\right)
  \!, \\
\end{align*}

\vspace{-0.45cm}

\noindent
and \vspace{-0.2cm}
\begin{align}\label{cor3.7}
\begin{array}{ll}
U^{\varphi}(x)
&
\begin{displaystyle}
=\Big[\cos\pi\nu(x)\Big] \exp\left( \beta + \int_{-\infty}^{+\infty}
   \log\frac{\sqrt{1+t^2}}{|x-t|}\,d\nu(t) \right) \,,
\end{displaystyle}
 \\[0.5cm]
V^{\varphi}(x)
&
\begin{displaystyle}
= \Big[\sin\pi\nu(x)\Big]\exp\left( \beta + \int_{-\infty}^{+\infty}
   \log\frac{\sqrt{1+t^2}}{|x-t|}\,d\nu(t) \right) \,,
\end{displaystyle}
       \end{array}
  \end{align}

\vspace{-0.45cm}
\noindent
respectively.

The following fact follows easily from Theorem~\sref{cor3}.


\vspace{0.1cm}
\begin{corollary}\label{cor4}
A function $z\cdot\varphi$ is universally starlike, i.e. $\varphi$ can
be represented in the form
\begin{equation}\label{cor4.1}
\varphi (z) = \exp\left( \int_0^1 \log\frac{1}{1-tz}\,d\mu(t) \right)\,,
\quad z\in\H\,,
\end{equation}
for some $\mu\in\mathcal{M}^{+}(\R)$ satisfying $\supp\mu\subseteq[0,1]$
and $\mu([0,1])=1$, if and only if $\varphi(0)=1$,
$\varphi\in \hol ((-\infty,1))$ and $\varphi\in\mathcal{P}_{\log}$.
\end{corollary}

\vspace{0.1cm}
\subsection{\texorpdfstring{Characterization of $\IM\mathcal{P}_{\log}$}{Characterization
of imaginary parts of functions defined in (1.17)}}\label{22a}
$\phantom{a}$ \\

 \vspace{-0.25cm}
Theorems~\sref{cor3} and~\sref{pr1} allow to describe the set of functions
$\IM\varphi$, when $\varphi$ runs over the class
$\varphi\in\mathcal{P}_{\log}\setminus\exp(\Delta_{\primP})$.

\begin{theorem}\label{th4}
Let $\varphi\in\mathcal{P}_{\log}$ and $\exp(\Delta_{\primP})$  be defined
in \eqref{9} and \eqref{2.4}, respectively.

For each function $\varphi\in\mathcal{P}_{\log}\setminus\exp(\Delta_{\primP})$
there exists a non-negative non-zero function
$v\in\bigcup_{p>1} L_p(\R,dx)$, satisfying for almost all
$-\infty<x_1<x_2<+\infty$ the inequality
\begin{equation}\label{th4.1}
\pi\cdot
\begin{vmatrix}
\widetilde{v}(x_1 ) & \widetilde{v}(x_2) \\ v(x_1) & v(x_2)
\end{vmatrix}
= \lim_{\varepsilon\downarrow0}
  \int_{\R\setminus(-\varepsilon,\varepsilon)}
  \begin{vmatrix}
  v (x_1+t) & v(x_2+t) \\ v(x_1) & v(x_2)
  \end{vmatrix}\,\frac{dt}{t}\geq 0\,,
\end{equation}
such that
\begin{equation}\label{th4.2}
\varphi(z)
= \frac{1}{\pi}\int_{-\infty}^{+\infty}\frac{v(t)\,dt}{t-z}\,,\quad z\in\H\,.
\end{equation}
Conversely, for each non-negative non-zero function
$v\in\bigcup_{p>1} L_p(\R,dx)$, satisfying \eqref{th4.1} for almost all
$-\infty<x_1<x_2<+\infty$, the function $\varphi$ defined in \eqref{th4.2}
belongs to the class $\mathcal{P}_{\log}\setminus\exp(\Delta_{\primP})$.
The representation \eqref{th4.2} is unique for every
$\varphi\in\mathcal{P}_{\log}\setminus\exp(\Delta_{\primP})$.

For every $\varphi\in\mathcal{P}_{\log}\setminus\exp(\Delta_{\primP})$,
\begin{equation}\label{th4.3}
\varphi(x+i0)
= \widetilde{v}(x) + iv (x)\quad \text{\rm $m$-a.e. on $\R$}\,,
\end{equation}
where $\widetilde{v}\in\bigcup_{p>1} L_p(\R,dx)$ and the functions $v$
and $\widetilde{v}$ are defined in \eqref{th4.2} and \eqref{2.5},
respectively.
\end{theorem}

\vspace{0.25cm}
\subsection{Universally starlike functions }\label{23}$\phantom{a}$ \\

\vspace{-0.35cm}
To apply the results of subsection~\sref{22} to the universally starlike
functions 
it is necessary to observe that according to \eqref{2.4}
\begin{equation}
\varphi(z):=
\exp\left( \int\limits_0^1 \log\frac{1}{1-tz}\,d\mu(t) \right)
\in\exp(\Delta_{\primP})
\end{equation}
if and only if
\begin{equation}\label{2.6}
\varphi(z) =
\begin{cases}
1\,,
& \supp\mu=\{0\}\,; \\
\dfrac{ a^{1-\mu(\{0\})} }{ (a-z)^{1-\mu(\{0\})} }\;,
& \supp\mu = \left\{ 0,{1}/{a} \right\}\,,\quad
  \mu(\{0\}) \in (0,1)\,, \\[-0.1cm]
& \mu\left( \left\{ {1}/{a} \right\} \right)
  = 1-\mu(\{0\})\,,\quad 1\leq a<+\infty\,; \\[0.2cm]
\dfrac{a}{a-z}\,,
& \supp\mu = \left\{ {1}/{a} \right\}\,,\quad
  \mu\left( \left\{ {1}/{a} \right\} \right)
  = 1\,,\quad 1\leq a < +\infty\,.
\end{cases}
\end{equation}

\medskip
This fact leads to the following consequence  of 
Corollary~\ref{cor4} and Theorem~\sref{th4}.

\vspace{0.1cm}
\begin{theorem}\label{th5}
If function $\Psi$ is universally starlike, then either
\begin{equation}\label{th5.1}
\frac{\Psi(z)}{z} \in \left\{
\dfrac{ a^{\theta} }{(a-z)^{\theta}}\; \mid\;
\theta\in[0,1]\,,\ 1\leq a<+\infty\,\right\}\,,
\end{equation}
or $\Psi(z)/z\in\bigcup_{p>1}\mathcal{H}_p(\H)$ and there exists a
function $v$ satisfying
\begin{subequations}
\begin{align}
\label{th5.2a} 
& v\in\bigcup_{p>1} L_p(\R,dx)\,,\quad v(x)=0\,,\quad
  x<1\,,\quad v(x)\geq0\,,\quad x\geq1\,; \\[0.3cm]
\label{th5.2b} 
& \int_{1}^{+\infty} \frac{v(t)}{t}\,dt = \pi\,, \\[0.3cm]
\label{th5.2c} 
& \pi\cdot
\begin{vmatrix}
\widetilde{v}(x_1) & \widetilde{v}(x_2) \\ v(x_1) & v(x_2)
\end{vmatrix}
= \lim_{\varepsilon \downarrow 0}
  \int_{ \R\setminus(-\varepsilon,\varepsilon)}
  \begin{vmatrix}
  v(x_1+t) & v(x_2+t) \\ v(x_1) & v(x_2)
  \end{vmatrix}
  \,\frac{dt}{t} \geq 0\,,
\end{align}
\end{subequations}

\vspace{0.25cm} \noindent
for almost all $1<x_1<x_2<+\infty$ such that
\begin{equation}\label{th5.3}
\frac{\Psi(z)}{z}
= \frac{1}{\pi} \int_{1}^{+\infty} \frac{v(t)\, dt}{t-z}\,,\quad z\in\H\,.
\end{equation}
Conversely, each function in the set
\begin{equation}\label{th5.4}
\left\{ \dfrac{a^{\theta}\cdot z}{(a-z)^{\theta}}\;\mid\;
\theta\in[0,1]\,,\ 1\leq a< +\infty \,\right\}\,,
\end{equation}
and each function of the type
\begin{equation}\label{th5.5}
\frac{z}{\pi} \int_{1}^{+\infty} \frac{v(t)\,dt}{t-z}\ \, ,
\end{equation}
where $v$ satisfies conditions \eqref{th5.2a}, \eqref{th5.2b} and
\eqref{th5.2c}, is universally starlike.
The function $\widetilde{v}(x)$ in \eqref{th5.2c} is defined by
formula \eqref{2.5}.
\end{theorem}

\vspace{0.25cm}
Observe that according to the results of Theorem~\sref{th4} two different
functions $v$ in \eqref{th5.5} correspond to two different universally
starlike functions.
Since
\begin{equation*}
\int\limits_{\R\setminus(-\varepsilon,\varepsilon)}
\begin{vmatrix}
v(x_1+t) & v(x_2+t) \\ v(x_1) & v(x_2)
\end{vmatrix}
\,\frac{d t}{t}
= \int\limits_{\varepsilon}^{+\infty}
 \left(\;
 \begin{vmatrix}
 v(x_1+t) & v(x_2+t) \\ v(x_1) & v (x_2)
 \end{vmatrix}
 +
 \begin{vmatrix}
 v(x_1) & v(x_2) \\ v(x_1-t) & v(x_2-t)
 \end{vmatrix}
 \;\right)
 \,\frac{dt}{t}\,,
\end{equation*}

\vspace{-0.3cm} \noindent
then according to \cite[Th.1.9, p.159]{kar} any function of the form
\begin{equation*}
v(x) = \chi_{[a,+\infty)}(x)\cdot e^{-\psi(x)}\,,\quad a\geq 1\,,
\end{equation*}
where $\psi$ is convex on $[a,+\infty)$ and
$\exp(-\psi(x))\in\bigcup_{p > 1} L_p([a,+\infty),dx)$ satisfies \eqref{th5.2c}.
Thus, the following assertion holds.

\vspace{0.25cm}
\begin{corollary}\label{cor5}
Let $a \geq 1$, $\gamma > 1$ and $\psi$ be a convex function on $[a,+\infty)$
such that
\begin{equation*}
\int_{a}^{+\infty} e^{-\gamma\cdot\psi(t)}\,dt < +\infty\,.
\end{equation*}
Put $b:=\int_{a}^{+\infty} t^{-1} e^{-\psi(t)}\,dt\,$.
Then the function
\begin{equation*}
\frac{z}{b}\,\cdot\int_{a}^{+\infty}\frac{e^{-\psi(t)}\;dt}{t-z}
\end{equation*}
is universally starlike.
\end{corollary}

We should also point out that Theorem~\ref{cor3} and Theorem~\ref{th3}
answer the question about finding a representation of all those non-negative
functions $v\in\bigcup_{p> 1} L_p(\R,dx)$ which satisfy \eqref{th4.1} for
almost all $-\infty<x_1<x_2<+\infty$.

\vspace{0.25cm}
\begin{corollary}\label{cor6}
A non-negative non-zero function $v$ belongs to $\bigcup_{p>1} L_p(\R,dx)$
and satisfies \eqref{th4.1} for almost all $-\infty<x_1<x_2<+\infty$
if and only if there exists a real number $\beta\in\R$ and a non-decreasing
function $\nu:\R\to[0,1]$ which has at least two growing points such that
\begin{equation}\label{cor6.1}
v(x) = \Big[\sin\pi\nu(x)\Big] \exp\left(
\beta + \int_{-\infty}^{+\infty} \log\lfrac{\sqrt{1+t^2}}{|x-t|}\,d\nu(t) \right)
\,,\quad x\in\R\,.
\end{equation}
\end{corollary}

\vspace{0.5cm}
\section{Preliminary Results}

 \vspace{0.35cm}
\subsection{Special properties of Pick functions}\las{31}$\phantom{a}$ \\

\vspace{-0.35cm}

Taking into account the special importance of the properties \eqref{3a}-\eqref{3c}
and \eqref{f1} for the theory of Pick functions, we give here independent simple
proofs of these relations.

For this purpose we recall a well-known result in the theory of $H_p$-spaces
(see \cite[p.70]{ko}, \cite[Th.3.1, p.57]{gar})
\begin{equation}\label{r1}
f\in H_p(\D)
\quad\Rightarrow\quad \exists\,\lim_{r\uparrow 1}
f(re^{i\theta})\in L_p([-\pi,\pi],\,d\theta)\,,\quad 0<p<+\infty\,,
\end{equation}
and observe that each $\Phi\in\PL$, $\Phi\not\equiv a$, $a\in\R$, is an
extension to $\hol(\C\setminus\R)$ of a unique function from the class
\begin{equation*}
\mathcal{P}_{+}
:= \left\{ \Phi\in\hol(\H) \mid \Phi(\H)\subseteq\H \right\}\,,
\end{equation*}
by means of the formula $\Phi(z) = \overline{\Phi(\overline{z})}$, $z\in-\H$.

\subsubsection{Proof of \eqref{3a} - \eqref{3c}.}\las{311}
Inequalities \eqref{3a} and \eqref{3b} are obviously valid for any constant
functions.
Thus, it is sufficient to prove them for all functions from $\mathcal{P}_{+}$
and $\log \mathcal{P}_{+}$, respectively.
Two fundamental properties of arbitrary function $\Phi\in\mathcal{P}_{+}$ are:
\begin{equation*}
f:=\log\Phi\in\mathcal{P}_{+}
\qquad\text{and}\qquad
\IM f=\arg \Phi\in(0,\pi)\,.
\end{equation*}
Now, introducing the change of variables
\begin{equation*}
\Phi_{\D} (w) :=
\Phi \left(i\,\frac{1-w}{1+w}\right)\, , \qquad
f_{\D} (w):= f \left(i\,\frac{1-w}{1+w}\right)
 = \log \Phi_{\D} (w)\,,\qquad
w\in\D\,,
\end{equation*}
we get
\begin{equation*}
\IM f_{\D} = \arg\Phi_{\D} \in (0,\pi)\,,
\end{equation*}
and so $\Phi_{\D}$ satisfies the conditions of a theorem of V. Smirnov (1928)
\cite{smi} (see \cite[p.409]{gol}, \cite[pp.253, 255, 444]{ost}).
According to the proof of this theorem given in \cite[p.114]{gar},
for arbitrary $\delta \in (0,1)$, we have:
\begin{align*}
i^{-\delta} \Phi_{\D}^{\delta}
&= |\Phi_{\D}|^{\delta} \big(
   \cos\delta (\pi/2-\arg \Phi_{\D} )
   - i\,\sin\delta ( \pi/2-\arg\Phi_{\D} ) \big)\in\hol(\D)\,, \\
\RE\big[ i^{-\delta} \Phi_{\D}^{\delta} \big]
&= |\Phi_{\D}|^{\delta} \cos\delta( \pi/2-\arg\Phi_{\D} )
   \geq |\Phi_{\D}|^{\delta} \cos( \delta\pi/2 )\,.
\end{align*}
Hence, by the mean value property of the harmonic function
$\RE\big[ i^{-\delta}\Phi_{\D}^{\delta}]$,
we have (see \cite[p.34]{hk}):
\begin{equation*}
\cos(\delta\pi/2)
\int_{-\pi}^{\pi} \left|\Phi_{\D}(r e^{i \theta})\right|^{\delta}\,d\theta
\leq \int_{-\pi}^{\pi}
\RE\left[ i^{-\delta}\Phi_{\D}^{\delta}(re^{i\theta})\right]\,d\theta
= 2\pi\RE\left[ i^{-\delta}\Phi_{\D}^{\delta}(0)\right]\quad\forall\,r\in[0,1)\,,
\end{equation*}

\vspace{-0.2cm} \noindent
i.e., $\Phi_{\D}\in H_{\delta}(\D)$ for any $\delta\in(0,1)$ and the proof
of the leftmost inclusion $\PL\subset H_\delta(\H)$ in \eqref{3c} is complete.

From $-1/\Phi\in\mathcal{P}_{+}$ it follows that we also have
$1/\Phi_{\D}\in H_{\delta}(\D)$, which in view of \eqref{r1} proves the
inequality \eqref{3a}.

Now, the evident inequality:
\begin{equation*}
x^{\delta} + x^{-\delta} \geq e^{\delta |\log x|}\,,\quad
x>0\,,\quad \delta \in (0,1)\,,
\end{equation*}
and the fact that $\Phi_{\D},\,1/\Phi_{\D}\in H_{\delta}(\D)$ imply
\begin{align}
\sup_{0\leq r<1} \int_{-\pi}^{\pi} e^{\ {\fo{\delta\left| f_{\D}(re^{i\theta}) \right|}}}\,d\theta
& \leq e^{\ {\fo{\delta\pi}}} \sup_{0\leq r<1} \int_{-\pi}^{\pi}
  e^{\ {\fo{\delta\left| \log|\Phi_{\D}(re^{i\theta})| \right|}}}\,d\theta \label{auxa21} \\
& \leq e^{\delta\pi} \sup_{0\leq r<1} \int_{-\pi}^{\pi}
  \left| \Phi_{\D}(re^{i\theta}) \right|^{\delta}\, d\theta
  + e^{\delta\pi} \sup_{0\leq r<1} \int_{-\pi}^{\pi}
  \left| \Phi_{\D}(re^{i\theta}) \right|^{-\delta}\,d\theta \notag \\
&= e^{\delta\pi}
   \left(\,\left\| \Phi_{\D} \right\|_{H_\delta(\D)}
   + \left\| 1/\Phi_{\D} \right\|_{H_\delta(\D)}\, \right) \notag \\
& < +\infty\,. \notag
\end{align}
Furthermore, the inequality
\begin{equation*}
x^{p} \leq \left( \frac{p}{\delta} \right)^p e^{\delta x}\,,\quad
x\geq0\,,\quad p>0\,,
\end{equation*}
applied to \eqref{auxa21} gives:
\begin{equation}\label{aux212}
\sup_{0\leq r<1} \int_{-\pi}^{\pi}
\left| f_{\D}(re^{i\theta}) \right|^p\,d\theta
\leq e^{\delta\pi} (p/\delta)^p\cdot
\left( \|\Phi_{\D}\|_{H_\delta(\D)} + \|1/\Phi_{\D}\|_{H_\delta(\D)} \right)
< +\infty\,,
\end{equation}
and so $f_{\D}\in H_p(\D)$ for all $p> 0$.
Thus, $f\in H_p(\H)$, $p>0$, which completes the proof of \eqref{3c}.

Finally,  \eqref{3b} follows from an application of the Fatou lemma
\cite[p.22]{rud} and \eqref{r1} to the inequality \eqref{auxa21},
\begin{equation*}
\int_{\R} \frac{e^{\ {\fo{\delta\,|f(x+i0)|}}}}{1+x^{2}}\,dx
= \int_{-\pi}^{\pi}
  e^{\ {\fo{\delta\left| f(\tan\lfrac{\theta}{2} + i0) \right|}}}\ \frac{d\theta}{2}
= \frac{1}{2} \int_{-\pi}^{\pi} 
  e^{\ {\fo{\delta\left| f_{\D}(e^{i\theta}) \right|}}}\,d\theta < +\infty\,,
\end{equation*}
and to \eqref{aux212} in a similar manner. \hfill$\square$

\subsubsection{Proof of \eqref{f1}.}\las{312}
According to \eqref{3b},
$\exp\big(\,|f(x+i0)|\,\big)\in L_{\delta}\big([-\Lambda,\Lambda],\,dx\big)$
for arbitrary  $\Lambda>0$ and $\delta\in(0,1)$.
Obviously,
\begin{equation}\label{3.1.1}
\frac{y}{(1+|x|)^{2}} + \frac{1}{y} \geq \frac{2}{1+|x|}\,,\quad
y>0\,,\quad x\in\R\,,
\end{equation}
and therefore, in view of \eqref{3b},
\begin{align*}
4\log(1+\Lambda)
&= \int_{-\Lambda}^{\Lambda} \frac{2}{1+|x|}\,dx \\[0.5cm]
&\leq \int_{-\Lambda}^{\Lambda} e^{\ {\fo{\pm\delta\cdot\RE f(x+i0)}}}\,dx
  + \int_{-\Lambda}^{\Lambda} \frac{e^{\ {\fo{\mp\delta\cdot\RE f(x+i0)}}}}{(1+|x|)^{2}}\,dx \\[0.5cm]
&\leq \int_{-\Lambda}^{\Lambda} e^{\ {\fo{\pm\delta\cdot\RE f(x+i0)}}}\,dx
  + \int_{\R} \frac{e^{\ {\fo{\delta\cdot|\RE f(x+i0)|}}}}{(1+|x|)^{2}}\,dx\,,
\end{align*}

\noindent \vspace{0.5cm}
from which we get the validity of \eqref{f1} by taking the limit as
$\Lambda\to+\infty$. \hfill$\square$

\vspace{0.2cm}
\subsection{Auxiliary lemma}\las{32}$\phantom{a}$ \\

\vspace{-0.1cm}

We first prove \eqref{5}.
Using \eqref{5ad} and the equality $\mu(1+0)=1$ we get
\begin{align*}
\int_{[0,1]} \log & \frac{1}{1-tz}\;d\mu(t)
+ \int_{[0,1]} \log\sqrt{1+t^2}\;d\mu(t) \\[0.2cm]
&= \int_{[0,1]} \log\frac{1-tz}{\sqrt{1+t^2}}\;d(1-\mu(t)) \\[0.2cm]
&= (1-\mu(t))\cdot\left.\log\frac{1-tz}{\sqrt{1+t^2}}\right|_{0-0}^{1+0} +
   \int_{[0,1]} (1-\mu(t))\;
   d\left( \lfrac{1}{2}\log(1+t^2) - \log(1-tz) \right) \\[0.2cm]
&= \int_0^1 (1-\mu(t))
   \left( \frac{t}{1+t^2} + \frac{z}{1-tz} \right)\;dt \\[0.2cm]
&= \int_0^1 \frac{(1-\mu(t))\,(t+z)}{(1+t^{2})\,(1-tz)}\;dt \\[0.2cm]
&= \int_1^\infty \frac{(1-\mu(1/t))\,(z+1/t)\;dt}{(1+1/t^2)\,(1-z/t)\,t^2} \\[0.2cm]
&= \int_1^\infty (1-\mu(1/t))\,\frac{(1+zt)\;dt}{(t^2+1)\,(t-z)} \\[0.2cm]
&= \int_1^\infty \left( \frac{1}{t-z}-\frac{t}{1+t^2} \right)\,
   (1-\mu(1/t))\;dt\,,
\end{align*}

\vspace{0.25cm} \noindent
as it was to be proved. \hfill$\square$

\medskip
In order to prove Theorem~\ref{th2} we will need the following assertion.


\vspace{0.25cm}
\begin{lemma}\las{lem4.1}
Let $\nu$ be a non-negative and non-decreasing function on $\R$ satisfying
\eqref{th2.2}.
Then
\begin{enumerate}
\item[{\rm{(1)}}]\quad
$\exists\,\lim_{t\to-\infty}\nu(t) =: \nu(-\infty) \geq 0$ {\rm{;}}
\vspace{0.15cm}\item[{\rm{(2)}}]\quad
$\exists\,\lim_{t\to+\infty} \nu(t)/t = 0$  {\rm{;}}
\vspace{0.15cm}\item[{\rm{(3)}}]\quad
\eqref{th2.4} holds.
\end{enumerate}
\end{lemma}

\begin{proof}\noindent
Statement  (1) is trivial.
To prove (2) suppose on the contrary that
\begin{equation*}
2 a := \limsp\limits_{t\to +\infty} \nu(t)/t > 0\,.
\end{equation*}
Then there exists a sequence of monotonically  increasing positive numbers
\begin{equation*}
1 < t_1 < t_2 < \dots < t_n < t_{n+1} < \dots \,,\quad
t_n\to+\infty\,, \quad n\to\infty\,,
\end{equation*}
such that
\begin{equation*}
\frac{\nu(t_n)}{t_n} \geq a\,,\quad n\geq1\,.
\end{equation*}
Since $\nu$ is a non-decreasing non-negative function, and in view of the
hypothesis \eqref{th2.2}, we have
\begin{align*}
+\infty
&> \int_1^{+\infty} \frac{\nu(t)\,dt}{t^2}
   \geq \int_{t_1}^{+\infty} \frac{\nu(t)\,dt}{t^2}
   = \sum_{n=1}^\infty \int_{t_n}^{t_{n+1}} \frac{\nu(t)\,dt}{t^2} \\
&\geq \sum_{n=1}^\infty \nu(t_n) \int_{t_n}^{t_{n+1}} \frac{dt}{t^2}
   \geq \sum_{n=1}^\infty a\cdot t_n
   \left( \frac{1}{t_n} - \frac{1}{t_{n +1}} \right)
=  a \sum_{n=1}^\infty \left( 1 - \frac{t_n}{t_{n +1}} \right)\,.
\end{align*}
Whence we have
\begin{equation*}
\sum_{n=1}^\infty \theta_n<+\infty\,,\quad
\theta_n := 1-\frac{t_n}{t_{n +1}}\in(0,1)\,,\quad n\geq1\,.
\end{equation*}
Thus,  $\lim_{n\to\infty}\theta_n=0$ and without loss of generality we may
assume that $\theta_n\leq\frac{1}{2}$ for all $n \geq 1$.
Then, iterating $t_{n+1}=\frac{t_n}{1-\theta_n}$ we get
\begin{equation*}
t_{n +1}
= \frac{t_n}{1-\theta_n}
= \frac{t_{n-1}}{(1-\theta_n)\,(1-\theta_{n-1})} = \dots
= \frac{t_1}{(1-\theta_n)\,(1-\theta_{n-1}) \dots (1-\theta_1)}\;,
\end{equation*}
and
\begin{equation*}
\log\prod_{k=1}^n \frac{1}{1-\theta_k}
= \sum_{k=1}^n \log\left( 1+ \frac{\theta_k}{1-\theta_k} \right)
\leq \sum_{k=1}^n \frac{\theta_k}{1-\theta_k}
\leq 2\sum_{k=1}^n \theta_k
\leq 2\sum_{k=1}^\infty \theta_k\,.
\end{equation*}
Therefore
\begin{equation*}
t_{n+1}\leq t_1\cdot e^{{\fo{2\sum_{k=1}^\infty \theta_k }}}\,,
\end{equation*}
in contradiction to the fact that $t_n\to+\infty$ as $n\to\infty$,
and (2) is verified.

\smallskip
In proving (3) it suffices to observe that \eqref{th2.2} for $A>1$ implies
\begin{align*}
+\infty
& > \int_1^{+\infty} \frac{\nu(t)\,dt}{t^2}
  \geq \int_1^A \frac{\nu(t)\,dt}{t^2}
  = - \int_1^A \nu(t)\;d\,\frac{1}{t}  \\
& = - \left. \nu(t) \cdot \frac{1}{t} \right|_1^A
    + \int_1^A \frac{d\nu(t)}{t}
   = \nu(1) - \frac{\nu(A)}{A} + \int_1^A \frac{d\nu(t)}{t}\,.
\end{align*}
Using the result of item (2) of this lemma and letting $A\to+\infty$ gives
\begin{equation*}
+\infty > \nu(1) + \int_1^{+\infty} \frac{d\nu(t)}{t}\;.
\end{equation*}
This completes the proof of Lemma~\sref{lem4.1}.
\end{proof}

\vspace{0.25cm}
\subsection{Particular case of the generalized H{\"o}lder inequality}\las{33}$\phantom{a}$ \\

\vspace{-0.1cm}

In order to prove Theorem~\sref{th3} we need to use a special  case of the
generalized H{\"o}lder inequality established in \cite[(11), p.460]{mrsh}.

Let $\N := \{1,2,\dots\}$, $\R^{+}:= [0,+\infty)$, $A\times B\subset\R^2$
the Cartesian product of $A,B\in\mathcal{B}(\R)$ \cite[p.8]{roy} and
$\mu_1\times\mu_2$ the product of $\mu_1,\mu_2\in\mathcal{M}^{+}(\R)$
\cite[p.304]{roy}.


\vspace{0.25cm}
\begin{lemma}\las{lem32}
Let $\Omega,Q\in\mathcal{B}(\R)$ be fixed and
$\rho,\omega\in\mathcal{M}^{+}(\R)$ such that
\begin{equation}\label{lem32.1}
0<\rho(\Omega)<1\,\quad 0<\omega(Q)<\infty\,.
\end{equation}
Let furthermore $\Psi:Q\times\Omega\to\R^{+}\cup\{+\infty\}$ be a Borel
measurable function \cite[Def.11.13, p.310]{rud1} such that
$\Psi(x,t)^{{\fo{\rho(\Omega)}}}\in L_1(Q,d\omega(x))$ for  all $t\in \Omega$ and
\begin{equation}\label{lem32.2}
{\rm (a)}\quad
\inf_{t\in\Omega} \Psi(x,t)>0\,,\ x\in Q\,; \qquad
{\rm (b)}\quad
D_{\Psi} := \sup_{t\in\Omega} \int_Q \Psi(x,t)^{{\fo{\rho(\Omega)}}}\,d\omega(x)<\infty\,.
\end{equation}
Then $\log\Psi(x,t)\in L_1(\Omega,d\rho(t))$ for almost all $x\in Q$ with
respect to the measure $\omega$ and
\begin{equation}\label{lem32.3}
\int_Q \exp\left( \int_\Omega
\log\Psi(x,t)\,d\rho(t) \right)\,d\omega(x)\leq D_{\Psi}\,.
\end{equation}
\end{lemma}

\begin{proof}
Let $\Phi(x,t):=\Psi(x,t)^{\rho(\Omega)}$,
$\gamma(x):=\min\{1,\inf_{t\in\Omega}\Phi(x,t)\}$
and $P (A):=\rho (A)/\rho(\Omega)$, $A\subset\Omega$, $A\in\mathcal{B}(\R)$.
In view of \eqref{lem32.2}(a), $\gamma (x)\in(0,1]$ for all $x\in Q$ and
therefore we have $\gamma:=\int_Q \gamma(x)\,d\omega(x)\in(0,\omega(Q)]$
by \eqref{lem32.1}.
We introduce the functions
\begin{equation*}
\Phi_n(x,t) := \min\{n,\Phi(x,t)\}\,,\qquad
\varphi_n(t) := \int_Q \Phi_n(u,t)\,d\omega(u)\,,\quad
x\in Q\,,\ t\in\Omega\,,\ n\in\N\,.
\end{equation*}
From $0<\gamma(x)\leq\Phi_n(x,t)\leq n$ it follows that for every
$x\in Q$ and $t\in\Omega$ the following inequalities hold:
\begin{equation*}
-\infty < \log \gamma \leq \log\varphi_n(t) \leq \log(n\cdot\omega(Q))\,,
\qquad -\infty < \log \gamma(x) \leq \log\Phi_n(x,t) \leq \log n\,.
 \end{equation*}
Thus, for each $n\in\N$ and $x\in Q$, $\log\Phi_n(x,t)$,
$\log\varphi_n(t)\in L_1(\Omega,\,dP(t))$ and hence it is possible to define
the finite-valued functions
\begin{align*}
h_n(x,t)
&:= \log\Phi_n(x,t) - \log\int_Q \Phi_n(u,t)\,d\omega(u)\,, \\
H_n(x)
&:= \int_\Omega h_n(x,t)\,dP(t) = \int_\Omega\log\Phi_n(x,t)\,dP(t)
    - \int_\Omega\left( \log\int_Q \Phi_n(u,t)\,d\omega(u) \right)\,dP(t)\,,
\end{align*}
for all $t\in\Omega$, $x\in Q$  and $n\in\N$.
Jensen's inequality \cite[Lem.6.1, p.34]{gar} applied to the exponential
function gives for all $x\in Q$ and $n\in\N$,
\begin{align*} 
\exp\left( \int_\Omega h_n(x,t)\,dP(t)\right)
& \leq \int_\Omega \exp\left( h_n(x,t) \right)\,dP(t) \\[0.2cm]
&= \int_\Omega \exp\left( \log\Phi_n(x,t) -
   \log\int_Q \Phi_n(u,t)\,d\omega(u) \right)\,dP(t) \\[0.2cm]
&= \int_\Omega \left[
   \frac{\Phi_n(x,t)}{\int_Q \Phi_n(u,t)\;d\omega(u)}\right]\,dP(t) \\[0.2cm]
&= \int_\Omega \left[ \frac{\Phi_n(x,t)}{\varphi_n(t)} \right]\,dP(t)\,,
\end{align*}
i.e.,
\begin{equation}\label{lem32.4}
e^{H_n(x)} \leq \int_{\Omega}
\left[ \frac{ \Phi_n(x,t)}{\int_Q \Phi_n(u,t)\,d\omega(u)}\right]\,dP(t)
= \int_\Omega \left[ \frac{\Phi_n(x,t)}{\varphi_n(t)} \right]\,dP(t)\,,
\quad x\in Q\,,\ n\in\N\,,
\end{equation}

\vspace{0.15cm} \noindent
where the function $\Phi_n(x,t)/\varphi_n(t)$, being positive and bounded
above by $n/\gamma$ on $Q \times \Omega$, belongs to
$L_1(Q\times\Omega,d(\omega\times P)(x,t))$.
Integrating \eqref{lem32.4} over $x\in Q$ with respect to the measure
$\omega$ we obtain by  Fubini's theorem \cite[p.307]{roy}
\begin{equation*}
\int_Q e^{H_n(x)}\,d\omega(x)
\leq \int_Q \int_{\Omega}
\frac{\Phi_n(x,t)}{\int_Q\Phi_n(u,t)\,d\omega(u)}\;dP(t)\,d\omega(x)
= \int_{\Omega} \dfrac{\int_Q \Phi_n(x,t)\,d\omega(x)}
  {\int_Q \Phi_n(u,t)\,d\omega(u)}\;dP(t) = 1\,,
\end{equation*}
from which it follows that
\begin{equation*}
\int_Q \exp\left\{ \int_\Omega \log\Phi_{n}(x,t)\,dP(t) - \int_\Omega\left(
\log\int_Q \Phi_{n}(u,t)\,d\omega(u)\right)\,dP(t) \right\}\,d\omega(x)
\leq 1 \, ,
\end{equation*}
i.e.,
\begin{equation*}
\int_Q \exp\left\{ \int_\Omega \log\Phi_{n}(x,t)\,dP(t)\right\}\;d\omega(x)
\leq \exp\left\{ \int_\Omega\left( \log\int_Q \Phi_{n}(u,t)\;d\omega(u)
\right)\,dP(t) \right\}\,.
\end{equation*}
Applying Jensen's inequality to the right-hand side of the last inequality
we obtain
\begin{align*}
\int_Q \exp & \left\{
  \int_\Omega \log\Phi_n(x,t)\,dP(t) \right\}\,d\omega(x)
  \leq \int_\Omega \exp \left( \log\int_Q
  \Phi_n(u,t)\,d\omega(u) \right) dP(t) \\[0.2cm]
&= \int_\Omega
  \left( \int_Q \Phi_n(u,t)\,d\omega(u) \right)\,dP(t)
  \leq \int_\Omega \left( \int_Q \Phi(u,t)\,d\omega(u) \right)\,dP(t) \\[0.2cm]
&= \int_\Omega
  \left( \int_Q \Psi(u,t)^{\rho(\Omega)}\,d\omega(u) \right)\,dP(t)
  \leq \int_\Omega D_\Psi\,dP(t) = D_\Psi\,.
\end{align*}

\vspace{0.15cm} \noindent
Summarizing, we have obtained the upper bound
\begin{equation}\label{lem32.5}
\int_{\Omega} \exp\left( 
\int_{\Omega} \log \Phi_n(x,t)\;dP(t) \right)\,  
d\omega(x) \leq D_{\Psi}\,,\quad n\in\N\,.
\end{equation}
Inequalities $0<\gamma(x)\leq\Phi_{n}(x,t)\leq\Phi_{m}(x,t)\leq m$,
which are valid for all $(x,t)\in Q\times\Omega$ and $m\geq n\geq 1$,
yield that $\Phi_{n}(x,t)$, $\log\Phi_{n}(x,t)$ and
$\exp\big(\int_{\Omega}\log\Phi_{n}(x,t)\,dP(t)\big)$, $n\in\N$,
are pointwise monotone non-decreasing sequences of finite-valued measurable
functions.
Hence, their limits as $n\to \infty$ exist, regardless of whether they are
finite or equal to $+\infty$.
Furthermore, the first two of these limits equal to $\Phi(x,t)$ and
$\log\Phi(x,t)$, respectively.
Now let us define:
\begin{equation*}
\psi(x):= \lim_{n\to\infty} \int_{\Omega}\log\Phi_{n}(x,t)\,dP(t)\,,\quad x\in Q\,.
\end{equation*}
It is clear that for each $x\in Q$
\begin{equation*}
-\infty
< \log\gamma(x)
= \int_\Omega \log\gamma(x)\,dP(t)
\leq \int_\Omega \log\Phi_n(x,t)\;dP(t)
\leq \lim_{n\to\infty} \int_\Omega \log\Phi_n(x,t)\;dP(t)
= \psi(x)\,.
\end{equation*}
Beppo Levi's theorem \cite[p.305]{kol} applied to the inequalities
\eqref{lem32.5} gives
\begin{align*}
\int_Q \exp\left( \psi(x) \right)\;d\omega(x)
&= \int_Q \exp\left(
   \lim_{n\to\infty}\int_{\Omega}\log\Phi_{n}(x,t)\,dP(t)
   \right)\,d\omega(x) \\[0.2cm]
&= \int_Q \lim_{n\to\infty}\exp\left(
   \int_{\Omega}\log\Phi_{n}(x,t)\,dP(t)
   \right)\,d\omega(x) \\[0.2cm] &
= \lim_{n\to\infty} \int_Q \exp\left(
   \int_{\Omega}\log\Phi_{n}(x,t)\,dP(t)
   \right)\,d\omega(x) \leq D_\Psi \ ,
\end{align*}
i.e., summarizing,
\begin{equation}\label{lem32.6}
\int_{\Omega}\exp\big( \psi (x) \big)\;d\omega(x) \leq D_{\Psi}\,,\quad x\in Q\,.
\end{equation}
Let $Q_\psi^\infty:=\big\{x\in Q\mid \psi(x)=+\infty\big\}$ and
$\widehat{Q}:=Q\setminus Q_\psi^\infty$.
Then (\ref{lem32.6}) implies that $\omega\left(Q_\psi^\infty\right)=0$.

Now, for each fixed $x\in\widehat{Q}$ one can again apply Beppo Levi's
theorem \cite[p.305]{kol} to the non-decreasing sequence of integrable
functions $n\geq\log\Phi_n(x,\cdot) > \log\gamma(x) > -\infty$, $n\in\N$,
satisfying
\begin{equation*}
-\infty < \log\gamma(x) \leq \int_\Omega \log\Phi_n(x,t)\;dP(t)
\leq \psi(x)  < +\infty\,,
\end{equation*}
to get $\log\Phi(x,t)\in L_1(\Omega,dP(t))$,
\begin{equation*}
\int_\Omega \log\Phi(x,t)\,dP(t)
= \int_\Omega \lim_{n\to\infty}\log\Phi_n(x,t)\,dP(t)
= \lim_{n\to\infty}\int_\Omega \log\Phi_n(x,t)\,dP(t)
= \psi(x) < +\infty\;,
\end{equation*}
and hence, $\log\Psi(x,t)\in L_1(\Omega,d\rho (t))$.
Thus,
\begin{equation*}
\psi(x) = \int_{\Omega} \log\Phi(x,t)\,dP(t)
= \int_{\Omega} \log\Psi(x,t)\,d\rho(t)\,.
\end{equation*}
Substitution of the last expression for $\psi$ in \eqref{lem32.6}
proves \eqref{lem32.3} and completes the proof of Lemma~\ref{lem32}.
\end{proof}


\vspace{0.45cm}
\section{Proofs of the main assertions}

\vspace{0.25cm}
\subsection{Proof of Theorem~\ref{th1} }

\subsubsection{Proof of item 1.}

For a non-negative function $V$, which is harmonic in the upper half plane
$\H$, the following well known representation holds (see \cite[p.107]{ko},
\cite[Th.II, p.20]{D}):
there exists a measure $\sigma\in\mathcal{M}^{+}(\R)$ with
\begin{equation} \label{3.1}
\int_{-\infty}^{+\infty} \frac{d\sigma(t)}{1+t^2} < +\infty\,,
\end{equation}
and a constant $\alpha\geq 0$ such that
\begin{equation}\label{3.2}
V(x,y) = \alpha y + \int_{-\infty}^{+\infty}
\frac{y}{y^2 + (x-t)^2}\;d\sigma(t)\,,\quad y>0\,,\quad x\in\R\,.
\end{equation}
Since for every $y > 0$ and
$x\in \R$ we have
\begin{align}
\frac{V(x +\varepsilon,y) - V(x,y)}{\varepsilon}
&= \frac{1}{\varepsilon} \left(
   \int_{-\infty}^{+\infty} \frac{y}{y^2+(x-t+\varepsilon)^2}\;d\sigma(t)
   -\int_{-\infty}^{+\infty} \frac{y}{y^2+(x-t)^2}\;d\sigma(t) \right) \label{3.2a} \\
&= \int_{-\infty}^{+\infty} \frac{1}{\varepsilon} \left(
   \frac{y}{y^2+(x-t+\varepsilon)^2}
   -\frac{y}{y^2+(x-t)^2} \right)\,d\sigma(t)\nonumber  \\
&= \int_{-\infty}^{+\infty}
   \frac{-y\,\varepsilon-2y\,(x-t)}
        {\big( y^2+(x-t+\varepsilon)^2 \big)\,\big( y^2+(x-t)^2 \big)}\;
   d\sigma(t)\nonumber  \\
& \longrightarrow \int_{-\infty}^{+\infty}
  \frac{-2y\,(x-t)}{\left( y^2+(x-t)^2 \right)^2}\; d\sigma(t)
  \quad\text{as $\varepsilon\to 0$\,.} \nonumber
\end{align}
Since by hypothesis $V_{x}(x,y)\geq0$ for $x\in\R$ and $y>0$, it follows that
\begin{equation} \label{3.3}
\frac{\partial}{\partial x} V(x,y)
= \int_{-\infty}^{+\infty} \frac{\partial}{\partial x}
  \left( \frac{y}{y^2 + (x-t)^2} \right)\;d \sigma(t) \geq 0\,,\quad
  y>0\,,\quad x\in\R\,.
\end{equation}
For any two real points $-\infty < x_1 < x_2 < +\infty$ let
\begin{equation*}
\Delta_{x_2}^{x_1}(x) 
= (x-x_1) \cdot\chi_{\raisebox{-4pt}{{\footnotesize $[\,x_1,(x_1+x_2)/2\,)$}}}(x) +
(x_2 - x) \cdot \chi_{\raisebox{-4pt}{{\footnotesize $[\,(x_1+x_2)/2, x_2\,]$}}}(x)\,,
\quad x\in\R\,.
\end{equation*}
It is obvious that $\Delta_{x_2}^{x_1}\in C(\R)$ and
\begin{align}
\int_{x_1}^{x_2} \Delta_{x_2}^{x_1}(x)
& \cdot
  \frac{\partial}{\partial x} \left( \frac{y}{y^2+(x-t)^2} \right)\,dx
  \label{3.4b} \\
& = \int_{x_1}^{x_2} \Delta_{x_2}^{x_1}(x)\;
  d\,\frac{y}{y^2 + (x-t)^2}
  = -\int_{x_1}^{x_2} \frac{y}{y^2+(x-t)^2}\;d\,\Delta_{x_2}^{x_1}(x)
  \nonumber \\
&= -\int_{x_1}^{\tfrac{x_1+x_2}{2}} \frac{y}{y^2+(x-t)^2}\;dx
   +\int_{\tfrac{x_1+x_2}{2}}^{x_2}\, \frac{y}{y^2 + (x-t)^2}\;dx \nonumber \\
&= \arctan\frac{x_2-t}{y} + \arctan\frac{x_1-t}{y}
   - 2\arctan\frac{\lfrac{x_1+x_2}{2}-t}{y} \ \ , \qquad\quad t\in\R\,.
  \nonumber
\end{align}
For any $y>0$, $-\infty<a<b<+\infty$ and $t\in\R$ a brief calculation
using the formulas \cite[4.4.34, p.80, 4.3.80, p.75]{ab}  gives
\begin{equation}\label{3.4a} 
0 < \arctan\frac{b-t}{y} - \arctan\frac{a-t}{y}
\leq
\begin{cases}
  \pi\,, & \quad \text{if $a-1\leq t\leq b+1$}\,, \\[0.2cm]
  y \cdot \dfrac{b-a+1+a^2+b^2}{1+t^2}\,,
         & \quad \text{if $t\notin[a-1,b+1]$}\,.
\end{cases}
\end{equation}
Thus, \eqref{3.3}, \eqref{3.4b},  \eqref{3.4a} and Fubini's theorem  yield
\begin{align}
0
&\leq \int_{x_1}^{x_2} \Delta_{x_2}^{x_1}(x) \cdot
   \frac{\partial }{\partial x}V(x,y) \;dx \label{3.4} \\
&= \int_{x_1}^{x_2}\Delta_{x_2}^{x_1}(x) \cdot
   \left( \int_{-\infty}^{+\infty} \frac{\partial}{\partial x}
   \left( \frac{y}{y^2+(x-t)^2} \right)\,d\sigma(t)\right)\,dx \nonumber \\
&= \int_{-\infty}^{+\infty} \left(
   \int_{x_1}^{x_2}\Delta_{x_2}^{x_1}(x) \cdot \frac{\partial}{\partial x}
   \left( \frac{y}{y^2+(x-t)^2} \right)\,dx \right)\,d\sigma(t)\nonumber \\
&= \int_{-\infty}^{+\infty} \left(
   \arctan\frac{x_2-t}{y} - \arctan\frac{(x_1+x_2)/2\,-\,t}{y} \right)\,d\sigma(t)
   \nonumber \\
& -\int_{-\infty}^{+\infty}  \left(
   \arctan\frac{(x_1+x_2)/2\,-\,t}{y} -
   \arctan\frac{x_1-t}{y}\right)\,d\sigma(t)\,. \nonumber
\end{align}
It is obvious that for arbitrary $-\infty < a < b < +\infty$
\begin{equation*}
\lim_{{\fo{y \downarrow 0}}}
\left( \arctan\frac{b-t}{y} - \arctan\frac{a-t}{y}\right)
= \begin{cases}
    0, & t < a\,, \\
    {\pi}/{2}, & t = a\,, \\
    \pi , &  t \in (a,b)\,, \\
   {\pi}/{2}, & t =b\,, \\
    0, & t > b\,.
\end{cases}
\end{equation*}
Due to \eqref{3.4a} we are able to apply Lebesgue's dominated convergence
theorem \cite[11.32, p.321]{rud} for \eqref{3.4} when $y \downarrow 0$ in
order to get the following inequality
\begin{align*}
0 \leq \lfrac{\pi}{2} & \left( \sigma(\{x_2\})
+ \sigma \left( \left\{ \lfrac{x_1+x_2}{2} \right\} \right) \right)
+ \pi\,\sigma \left( \left( \lfrac{x_1+x_2}{2}, x_2 \right) \right) \\[0.2cm]
&- \lfrac{\pi}{2} \left( \sigma( \{x_1\} )
+ \sigma \left( \left\{ \lfrac{x_1+x_2}{2} \right\} \right) \right)
- \pi\,\sigma \left( \left( x_1,\lfrac{x_1+x_2}{2} \right) \right)\,,
\end{align*}
from which we get
\begin{align*}
0 \leq \sigma(x_2+0)
& - \sigma(x_2-0)  - \sigma(x_1+0) + \sigma(x_1-0) \\[0.2cm]
& + 2 \sigma(x_2-0) - 2 \sigma\left(\lfrac{x_1+x_2}{2}+ 0\right) -
    2 \sigma\left(\lfrac{x_1+x_2}{2}- 0\right) + 2 \sigma(x_1+0) \\[0.2cm]
&= \sigma(x_2+0)
   + \sigma(x_2-0) + \sigma(x_1+0) + \sigma(x_1-0) \\[0.2cm]
& \qquad\quad - 2\left( \sigma\left(\lfrac{x_1+x_2}{2}+ 0\right)
   + \sigma\left(\lfrac{x_1+x_2}{2} - 0\right)\right)\,.
\end{align*}
As explained in subsubsection~\ref{21}, the function $\sigma(x)$
can be chosen to satisfy \eqref{aux2} and therefore the latter
inequality implies
\begin{equation}\label{3.5}
\sigma\left(\frac{x_1+x_2}{2} \right) \leq \frac{\sigma(x_1)+\sigma(x_2)}{2}\,,
\quad -\infty < x_1 < x_2 < +\infty\,.
\end{equation}
But $\sigma (x)$, being  a nondecreasing function on $\R$, is  bounded above
by $\sigma(A)$ in $(-\infty,A]$ for every $A>0$.
Thus, in view of \cite[111, p.91]{har}, the property \eqref{3.5} means  that
$\sigma (x)$ is convex i.e.,
\begin{equation*}
\sigma\left( \theta x_1 + (1-\theta) x_2 \right)
\leq \theta \sigma(x_1) + (1-\theta) \sigma(x_2)\qquad
\forall\,x_1,x_2\in\R\quad \forall\,\theta\in [0,1]\,.
\end{equation*}
Then, on each segment $[-A,A]$, $A>0$, the function $\sigma$  satisfies a
Lipschitz condition and is thus absolutely continuous
(see \cite[Th.1, p.230]{nat2}, \cite[p.244]{nat}).
Furthermore, by \cite[111, p.91]{har},  the one-sided derivatives of
$\sigma (x)$ exist at every point $x\in\R$
\begin{equation}\label{3.5a}
\nu(x+0) := \lim_{h \downarrow 0} \lfrac{\sigma(x+h) - \sigma(x)}{h}\,,\qquad
\nu(x-0) := \lim_{h \downarrow 0} \lfrac{\sigma(x) - \sigma(x-h)}{h}\,,
\end{equation}
and moreover,
since the function $\sigma$ is convex and nondecreasing on $\R$,
its derivative $\nu$ satisfies:
\begin{equation*}
0\leq\nu(x-0)\leq\nu(x+0)\leq\nu(y-0)\leq\nu(y+0)\quad
\text{for}\quad -\infty<x<y< +\infty\,.
\end{equation*}
Putting $\nu(x)$ to be an arbitrary number in the interval
$[\nu(x-0),\nu(x+0)$, $x\in\R$, we get a non-negative nondecreasing function
$\nu $ on  $\R$ which according to \cite[Prop.4.9, p.202; (4.27), p.203]{mp}
satisfies
\begin{equation*}
\sigma(E) = \int_{E}\nu(x)\,dx\,,\quad
\text{for any bounded Borel set $E\subset\R$}\,.
\end{equation*}
The change of variable theorem \cite[Cor 4.3, p.214]{mp} applied to
\eqref{3.1} and \eqref{3.2}, gives \eqref{th1.2} and \eqref{th1.3},
respectively.

\smallskip
We now prove the inverse statement. For any $x\in \R$ and $h> 0$ the property
\begin{equation*}
 \alpha y + \int_{-\infty}^{+\infty} \frac{y}{y^2 + (x-t)^2} \ \nu (t) \ d t \leq
\alpha y  +  \int_{-\infty}^{+\infty} \frac{y}{y^2 + (x+h -t)^2} \ \nu (t) \ d t \ ,
\end{equation*}
or, what is the same,
\begin{equation*}
  \int_{-\infty}^{+\infty} \frac{y}{y^2 + (x-t)^2} \ \nu (t) \ d t \leq
  \int_{-\infty}^{+\infty} \frac{y}{y^2 + (x+h -t)^2} \ \nu (t) \ d t =
 \int_{-\infty}^{+\infty} \frac{y}{y^2 + (x -t)^2} \ \nu (t+h) \ d t \ ,
\end{equation*}
follows from the evident inequality
\begin{equation*}
 \int_{-\infty}^{+\infty} \frac{y}{y^2 + (x -t)^2} \  ( \nu (t+h) - \nu (t) ) \ d t
\geq 0 \ .
\end{equation*}
Since \eqref{3.2} is unique (see \cite[Th.II, p.20]{D}) and the Radon-Nikodym
derivative $\nu = d \sigma / d m$ is determined almost everywhere on $\R$ we
have  the uniqueness of \eqref{th1.3}  provided that all functions in the
class $\aleph_{\,\nu}$, defined in \eqref{aux4}, are identified.

\vspace{0.25cm}
\subsubsection{Proof of item 2.}

Since $V_x (x,y)$ is a non-negative function harmonic  in $\H$ it
follows from the mean value
property of harmonic functions \cite[(1.5.4), p.34]{hk} that either
$V_x (x,y)> 0$ or $V_x (x,y)= 0$ everywhere in $\H$.

Assume  that  $V_x (x,y) \equiv 0$ in $\H$.
Then, the leftmost integral in \eqref{3.4} would be zero and
hence equality would be attained in \eqref{3.5} for all
$-\infty < x_1 < x_2 < +\infty$.
As $\sigma (x)$ is nondecreasing and finite on $\R$,
this means by \cite[111, p.91]{har} that $\sigma (t) = a t + b$
for some $a \geq 0$ and $b \in \R$.  Using \eqref{3.2} and
\eqref{3.5a} we get $V (x,y) = \alpha y  +  a \cdot \pi$ and
$\nu (x)  \equiv a $ for all  $y > 0$ and $x \in \R$,
 as stated in Theorem~\sref{th1}.
Conversely, if $\nu  \equiv a \geq 0$ then we get from \eqref{th1.3}
that $V (x,y) = \alpha y  +  a \cdot \pi$ and so  $V_{x} (x, y) \equiv 0$.
Finally, if for some $\alpha, a \geq 0$ we have $V (x,y) = \alpha y  +
 a \cdot \pi$ for all $y > 0$ and $x \in \R$, then the uniqueness of
the representation \eqref{th1.3} proved above gives  $\nu  \equiv a $.
The proof is complete. \hfill $\Box$

\vspace{0.5cm}
\subsection{Proof of Theorem~\ref{th2}}

\subsubsection{Proof of item 1.}

If $\Phi \in  \mathcal{P}_{\int} \cap \mathcal{P} \subset \mathcal{P}$,
then according to \cite[Th.II, p.20]{D} it has a unique canonical
representation of the form \eqref{2} and $V = V^{\Phi}$ satisfies the
conditions of Theorem~\sref{th1}. Since the measure $\sigma$ is completely
determined by the function  $V^{\Phi}$    \cite[Th.II, p.21]{D}, applying
Theorem~\ref{th1} we get the representation \eqref{th2.1} for $\Phi$.
Conversely, a function $\Phi$ with the integral representation
\eqref{th2.1} has an imaginary part $V$ with the representation
\eqref{th1.3} and therefore $\Phi \in  \mathcal{P}_{\int} \cap \mathcal{P}$,
due to the results of Theorem~\sref{th1}.

In view of \eqref{aux4} and \eqref{aux11} it is possible to replace
$\nu$ in \eqref{th2.1} and \eqref{th2.3}  by an arbitrary function  from
the class $\aleph_{\,\nu} $  defined in \eqref{aux4}.
Since \eqref{th2.1}  is a particular case of the unique representation
\eqref{2} and  \eqref{th2.3} is equivalent to \eqref{th2.1} we get
the uniqueness of \eqref{th2.1} and \eqref{th2.3} provided that  all
functions in $\aleph_{\,\nu}$ are identified.

The representation \eqref{th2.3} can be easily established integrating
\eqref{th2.1} by parts .
Since \eqref{th2.4} holds true because of Lemma~\ref{lem4.1},(3) and since
for $z\in\H$ and $|t|\to+\infty$ the integrand in \eqref{th2.3} can be
estimated as
\begin{align*}
\log\frac{t-z}{\sqrt{1+t^2}}
&= \log\frac{ {\rm sign}(t)-z/|t| }{\sqrt{1+1/t^2}} \\
&= \log\left( {\rm sign}(t)-z/|t| + O(1/t^2) \right)
 = \begin{cases}
    -z/|t| + O(1/t^2)\,,        & t>0\,, \\
    -i\pi + z/|t| + O(1/t^2)\,, & t<0\,,
   \end{cases}
\end{align*}
it follows that the integral  in \eqref{th2.3} converges absolutely.
Thus, an application of Lemma \ref{lem4.1}, (1)-(2) to \eqref{th2.1}
gives
\begin{align*}
& \int_{-\infty}^{+\infty} \left( \lfrac{1}{t-z}-\lfrac{t}{1+t^2} \right)
  \,\nu(t)\,dt
  = \int_{-\infty}^{+\infty} \nu(t)\,
    d\left( \log(t-z) - \lfrac{1}{2}\log(1+t^2) \right) \\
&\qquad
= \int_{-\infty}^{+\infty} \nu(t)\,d\log\lfrac{t-z}{\sqrt{1+t^2}}
= - \int_{-\infty}^{+\infty} \log\lfrac{t-z}{\sqrt{1+t^2}}\,d\nu(t)
  + \lim_{T\to +\infty} \nu(t) \log\lfrac{t-z}{\sqrt{1+t^2}}
    \bigg|_{-T}^{T} \\
&\qquad
= - \int_{-\infty}^{+\infty} \log\lfrac{t-z}{\sqrt{1+t^2}}\,d\nu(t) \\
&\qquad\quad\ + \lim_{T\to +\infty} \left[ \nu(T)
    \left( -\lfrac{z}{T} + O \left( \lfrac{1}{T^2} \right) \right)
    - \nu(-T) \left( -i\pi + \lfrac{z}{T}
    + O\left( \lfrac{1}{T^2} \right) \right) \right] \\
&\qquad
= i\pi\nu(-\infty)
  + \int_{-\infty}^{+\infty} \log\lfrac{\sqrt{1+t^2}}{t-z}\,d\nu(t)\,,\quad
  z\in\H\,,
\end{align*}
which proves  \eqref{th2.3} and its equivalence to  \eqref{th2.1}
for arbitrary function $\nu$ in question.

\vspace{0.25cm}
\subsubsection{Proof of item 2.}

In view of
\begin{equation}\label{4.5}
- \log (t-z) = - \log \sqrt{y^{2} + (t-x)^{2}} + i \left(\lfrac{\pi}{2}
- \arctan\lfrac{t-x}{y}\right)\,,\quad z=x+iy\in\H\,,
\end{equation}
equations \eqref{th2.5} and \eqref{th2.6} follow easily from \eqref{th2.3}.

To calculate the boundary values of $V^{\Phi} (x,y)$ we fix an arbitrary
$x,\theta\in\R$ and let $y$ be only in the interval $(0,1)$.
Observe that
\begin{equation}\label{4.1}
\lim_{y\downarrow 0} \left(
\lfrac{\pi}{2} - \arctan \lfrac{t-x - \theta y }{y} \right)
= \begin{cases}
  \pi, & \hbox{$t<x$;} \\[0.2cm]
  \lfrac{\pi}{2} + \arctan \theta , & \hbox{$t=x$;} \\[0.2cm]
  0 , & \hbox{$t>x$.}
  \end{cases}
\end{equation}
From  $0 \leq \lfrac{\pi}{2} -\arctan \lfrac{t-x- \theta y}{y} \leq \pi $
and \cite[4.4.42, p.81]{ab}
\begin{equation*}
\lfrac{\pi}{2} - \arctan \lfrac{t-x- \theta y}{y}
=  \lfrac{y}{t-x- \theta y} + O (\lfrac{y^{3}}{|t-x|^{3}}) \quad
\text{as}\quad t \to + \infty \ ,
\end{equation*}
it follows that there exists a finite positive constant
$C_1 \geq 2 (|x| + |\theta|)$ such that
\begin{equation*}
  0 \leq \lfrac{\pi}{2} -\arctan \lfrac{t-x- \theta y}{y} \leq \pi \cdot
\chi_{(-\infty , C_1]} (t) + \lfrac{2 }{t} \cdot \chi_{( C_1 ,
 +\infty) } (t)  \ , \ \ t \in \R \ , \
y \in (0,1) \ .
\end{equation*}
According to \eqref{th2.4} this means that
$\lfrac{\pi}{2} -\arctan  \lfrac{t-x- \theta y}{y} \in L_{1}(\R, d \nu (t))$,
so that Lebesgue's dominated convergence theorem \cite[p.26]{rud}
guarantees the passage to the limit in \eqref{th2.6} for $y\downarrow0$.
With the help of \eqref{aux2} we then obtain
\begin{align*}
\lim_{y \downarrow 0}
& V^{\Phi} (x +\theta y,y) \\
&= \pi \nu (-\infty) + \pi \nu ((-\infty, x)) +
   \left(\lfrac{\pi}{2} + \arctan \theta \right) \nu (\{x\}) \\
&= \pi \nu (-\infty) + \pi \left( \nu (x - 0) - \nu (-\infty) \right) +
   \lfrac{\pi}{2} \left( \nu (x + 0) - \nu(x - 0) \right) +
   \nu (\{x\}) \arctan \theta \\
&= \pi \lfrac{\nu (x + 0) + \nu (x - 0)}{2} +
   \nu (\{x\}) \arctan \theta = \pi \nu (x) + \nu(\{x\})\arctan \theta\,.
\end{align*}
This proves \eqref{th2.8} for every $x \in \R$ and
$\mathcal{D} (V^{\Phi}) \subset \R \setminus D_{\nu}$.

Fix now an arbitrary $x \in \mathcal{D} (U^{\Phi})$ and calculate the limit
of $U^{\Phi} (x,y)$ as $y \downarrow 0$,
which according to the definition of $\mathcal{D} (U^{\Phi})$
exists, is finite, and is equal to $U^{\Phi} (x)$.
We observe that the function $\log\sqrt{\lfrac{1 + t^2}{y^2 + (t-x)^2}}$
increases monotonically to the limit function
$\log \lfrac{\sqrt{1 + t^2}}{|t-x|}$ as $y \downarrow 0$.
But the right-hand side integral in \eqref{th2.5} is bounded above for
all $y \in (0, 1)$ and so one can apply the  Beppo Levi theorem
\cite[p.305]{kol} to get
$\log\lfrac{\sqrt{1 + t^2}}{|t-x|} \in L_1 (\R, d \nu)$
and the validity of \eqref{th2.7} for all
$x \in \mathcal{D} (\Phi) \subset \mathcal{D} (U^{\Phi})$.

\vspace{0.45cm}
\subsection{Proof of Corollary~\ref{corr1}}$\phantom{a}$ \\


For every $z = x + i y $, $x , t \in \R$, $y > 0$ and arbitrary
$| \varepsilon | < y / 2$  we have
\begin{equation*}
\left| \frac{\varepsilon}{t - z }\right|
= \left| \frac{\varepsilon}{t - x - i y }\right| < \frac{1}{2}\,,
\end{equation*}
and due to \eqref{th2.4},

\vspace{-0.25cm}
\begin{equation*}
- \frac{1}{\varepsilon}\,\log \left(1-\frac{\varepsilon}{t-z}\right)
\ \longrightarrow\ \frac{1}{t- z}
\stackrel{\eqref{th2.4}}{\in}
L_1 \left(\R , d \nu (t)\right)
\quad\text{as $\varepsilon\to 0$}\,.
\end{equation*}

\vspace{0.25cm} \noindent
Since by \cite[p.68, 4.1.38, 4.1.33]{ab} and \eqref{th2.4}

\vspace{-0.25cm}
\begin{equation*}
\frac{1}{|\varepsilon|} \cdot \left| \log
\left(1- \frac{\varepsilon}{t - z}\right)\right|
\stackrel{(4.1.38)}{\leq}
\frac{1}{|\varepsilon|} \cdot \left| \log \left( 1+
\frac{\left| \frac{\varepsilon}{t-z} \right|}
     {1 - \left| \frac{\varepsilon}{t-z} \right|}
\right) \right|
\stackrel{(4.1.33)}{\leq}
\frac{2}{\left| t-z \right|}
\stackrel{\eqref{th2.4}}{\in}
L_1 \left(\R , d \nu (t)\right) \ ,
\end{equation*}
we can deduce from \eqref{th2.3} and  Lebesgue's dominated convergence
theorem \cite[p.26]{rud}  that

\vspace{-0.25cm}
\begin{align*}
\frac{ \Phi (z+ \varepsilon) - \Phi (z)}{\varepsilon} - \alpha
&= \frac{1}{\varepsilon} \left(
   \int_{-\infty}^{+\infty} \log \frac{\sqrt{1+t^2}}{t-z-\varepsilon}\;d\nu(t)
   -\int_{-\infty}^{+\infty} \log \frac{\sqrt{1+t^2}}{t-z}\;d\nu(t)\right) \\[0.2cm]
&= \int_{-\infty}^{+\infty} \frac{1}{\varepsilon}
   \left( \log \frac{\sqrt{1+t^2}}{t-z-\varepsilon}
   -\log \frac{\sqrt{1+t^2}}{t-z}\,\right)\;d\nu(t) \\[0.2cm]
&= -\int_{-\infty}^{+\infty} \frac{1}{\varepsilon}
   \log \left( 1-\frac{\varepsilon}{t-z} \right)\;d\nu(t) \\[0.2cm]
&\longrightarrow \int_{-\infty}^{+\infty} \frac{d\nu(t)}{t-z}
\quad\text{ as $\varepsilon\to 0$}\,,
\end{align*}

\vspace{0.15cm} \noindent
which proves the validity of  \eqref{th2.9}.

\vspace{0.15cm}
Assume that $\Phi\in\mathcal{P} \cap \primP$ is a constant function.
Then the definition of the class $\mathcal{P} \cap \primP$
allows $\Phi$ to be any constant of the form $a+ib$, $a\in\R$,
$b \geq 0 $.
Any of such constants can be represented by the formula \eqref{th2.3}
with $\alpha = 0$, $\beta = a$ and $\nu(x) \equiv \nu(-\infty) = b/\pi$.
Since the representation \eqref{th2.3} is unique, there is no 
other choice 
for $\alpha$, $\beta$ and $\nu$.

\vspace{0.15cm}
If $\nu$ is a constant function then $\supp \nu = \{\emptyset\}$ and
\eqref{th2.3} implies directly that
$\Phi(z) = \alpha z + \beta + i \pi \cdot \nu(-\infty)$ and
$\Phi^{\,\prime}(z) \equiv \alpha $, $\IM\Phi^{\,\prime}(z) \equiv 0$.

\vspace{0.15cm}
If $\Phi$ is not a linear function, then according to what was just proved,
$\nu$ in \eqref{th2.3} is a non-constant function.
Thus, $\supp \nu \neq \{\emptyset\}$ and for $z = x+iy \in \H$
\eqref{th2.6} implies that
$V(x,y) \geq \int_{\R}(\lfrac{\pi}{2} - \arctan\lfrac{t-x}{y})\,d\nu(t)>0$
because the integrand is strictly positive on the whole $\R$ and by virtue
of \eqref{th2.9},
$V_{x}(x,y) \geq \int_{\R} \lfrac{y}{y^2 + (x-t)^2}\ d\,\nu(t) > 0$
due to the same reason.
This completes the proof of Corollary~\sref{corr1}.

\vspace{0.5cm}
\subsection{Proofs of Theorem~\ref{cor1} and Corollary~\ref{cor2}}

\subsubsection{Proof of Theorem~\ref{cor1}.}\las{61}

Note first  that any function
$f\in \primP  \cap  \log \mathcal{P}$ belongs to the
intersectio of the classes $\primP  \cap   \mathcal{P}$ and
$\log \mathcal{P}$.
Therefore the results of Theorem~\ref{th2} and the representations
\eqref{2} and \eqref{3} are valid for $f$.
Uniqueness of the representation \eqref{2} for $f$ together with the
possibility of writing it in both forms \eqref{th2.1} and \eqref{3}
implies immediately the representation \eqref{cor1.2} with $\nu$
satisfying \eqref{cor1.1}.
All other statements of Theorem~\sref{cor1} except the last three ones
are direct consequences of the corresponding statements of
Theorem~\sref{th2}.
The last three statements of Theorem~\sref{cor1} follow from
\eqref{cor1.1}, Theorem~\sref{th2}, 4) and the additional remark that
a linear function belongs to $\log \mathcal{P}$ if and only if it is
the constant one, as it follows from the inequalities
$0 \leq \IM \varphi (z) \leq \pi $, $z \in \H$,
for arbitrary  $\varphi \in \log \mathcal{P}$.

\subsubsection{Proof of Corollary~\ref{cor2}.}\las{62}

If $f (z) = \int_{0}^{1} \log \frac{1}{1 - t z} \ d \mu (t)$,
with $\mu \in \mathcal{M}^{+}(\R)$ satisfying $\supp\mu \subseteq [0, 1]$
and $\mu( \R ) = 1 $,  then obviously $f \in \hol((-\infty,1))$,
$f (0) = 0$  and, in view of \eqref{5} and  \eqref{cor1.2},
$f \in \primP  \cap  \log \mathcal{P}$.

Conversely, $f \in \primP \cap  \log \mathcal{P} \subset \mathcal{P}$
and $f \in \hol((-\infty,1))$ imply by \eqref{2} and \eqref{a2}
(see also \cite[p.18]{D}), that $\IM f (x+i0) = 0 $
for any $x < 1$.
Due to \eqref{cor1.7} this means that the representations \eqref{cor1.2}
and \eqref{cor1.3} are valid  for the function $f$ with $\nu (x) = 0$ for
all $x < 1$.
Thus, using the change of variable theorem for the Lebesgue-Stieltjes
integral \cite[Th.6.2.1, p.97]{cart} we can rewrite \eqref{cor1.3} as
follows
\begin{equation}\label{5.1}
f (z) = \beta + \int_{[1, +\infty)} \log\frac{\sqrt{1+t^2}}{t-z}\;d\nu(t)
= \beta + \int_{(0, 1]} \log\frac{\sqrt{1+t^2}}{1-zt}\;
  d\left(1-\nu(1/t)\right)\,.
\end{equation}
Put now
\begin{equation}\label{5.3}
\mu (x) :=
\begin{cases}
 0\,,            & x<0\,, \\
 1 - \nu(1/x)\,, & x>0\,,
\end{cases}
\end{equation}
and $\mu (0) := 1/2 - \nu (+\infty)/2 $.
Then obviously $\mu \in M^{\uparrow } (\R)$ and if $\nu$ satisfies
\eqref{aux2}, the function $\mu$ satisfies the same.
It follows from $0 < \nu (\R) =  \nu (+\infty) \leq 1$ that
$\mu (0+0) = 1 - \nu (+\infty) \in [0,1) $.
For the Lebesgue-Stieltjes measure $\mu \in \mathcal{M}^{+} (\R)$ induced
by this function $\mu  \in M^{\uparrow } (\R)$ we have
$\mu (\R) = \mu (+\infty) - \mu (-\infty) = (1- \nu (0+0)) - 0 = 1$
and the value of
$\mu (\{0\}) = \mu (0+0) - \mu (0-0) = 1 - \nu (+\infty)\in [0, 1)$
can be nonzero.
But since the integrand $\log \frac{\sqrt{1 + t^2}}{1 - z t}$
in \eqref{5.1} vanishes at $t = 0$, we get
\begin{equation}\label{5.1a}
f(z) = \beta  + \int_{[0, 1]} \log \frac{\sqrt{1+t^2}}{1-zt}\;d\mu(t)\,.
\end{equation}
For any $z \in \H$ the absolute values of $\log\lfrac{\sqrt{1+t^2}}{1-zt}$
and $\log\sqrt{1+t^2}$ are bounded above in $t\in[0,1]$ and therefore
$f(0) = 0$ and \eqref{5.1a} yield
$\beta = - \int_{0}^{1} \log \sqrt{1 + t^2} \ d \mu (t)$
and
\begin{equation}\label{5.2}
f(z) = \int_{[0, 1]} \log \frac{1}{1 - z t}\;d\mu(t)\,,
\end{equation}
which finishes the proof of \eqref{cor2.1}.

Finally, if $f$ is identically constant, then it follows from $f(0)=0$
that $f \equiv 0$ and by Theorem~\ref{cor1} $\nu \equiv 0 $ which gives
by virtue of \eqref{5.3}
$\mu (x) 
= \chi_{\raisebox{-3pt}{{\footnotesize $(\,0,+\infty\,)$}}}(x) 
+ (1/2)\chi_{\raisebox{-3pt}{{\footnotesize $\{0\}$}}}(x)$, $x\in\R$, where
$\chi_{\raisebox{-3pt}{{\footnotesize $A$}}}(x)$ 
denotes the characteristic function of $A\subset\R$.
This obviously means that $\supp \mu = \{0\}$ and $\mu (\{0\})=1$
what was to be proved.


\vspace{0.4cm}
\subsection{Proof of Theorem~\ref{th3}} \mbox{}

\vspace{0.1cm}
\subsubsection{}
\
According to the definition of ${\Delta_{\primP}}$ made before \eqref{2.3},
we obviously have that $f \notin {\Delta_{\primP}}$ if and only if the
support of $\nu$ consists of at least two different points.

We assume now that $f \notin {\Delta_{\primP}}$, introduce  the function
\begin{equation*}
\varphi (x; t) := \frac{\sqrt{1+t^2}}{|t-x|}\,,\quad x,\,t\in\R\,.
\end{equation*}
and prove first the following crucial lemma.

\vspace{0.25cm}
\begin{lemma}\las{l6.1}
Let $\nu \in \ms$  with $\nu (\R) \in (0,1]$ and
\begin{equation*}
\alpha_{\nu} (x) :=
\exp\left(
\int_{-\infty}^{+\infty} \log\lfrac{\sqrt{1+t^2}}{|\,t-x\,|}\,d\nu(t)
\right)\,,\quad x\in\R\,.
\end{equation*}
Then for every $a\in{\rm{supp}}\,\nu$ with $\nu(\{a\})<1$ and every
$p\in\left(\,1,{1}/{\nu(\{a\})}\,\right)$ there exists
$\varepsilon = \varepsilon (a, p)> 0$ such that
\begin{equation}\label{6.8}
\int_{a-\varepsilon }^{a+\varepsilon }
\alpha_{\nu}(x)^p\,dx < + \infty\,.
\end{equation}
\end{lemma}

\begin{proof}\quad
If $a \notin \supp \nu$ then $\nu (\{a\}) = 0 $ and there exists
$\varepsilon>0$ such that $\nu([a-2\varepsilon,a+2\varepsilon])=0$.
Since $\alpha_{\nu} (x)$ is uniformly bounded above on
$[a-\varepsilon , a+\varepsilon ]$ we get validity of \eqref{6.8}
for arbitrary $ p> 0$ as was to be proved.

Assume that $a \in \supp \nu$.  Denote $\theta := \nu (\{a\}) \in [0, 1)$
and fix $p > 1$ with $p \cdot \theta < 1$.
Since
\begin{equation*}
\lim_{\varepsilon\downarrow 0} \nu\big((a-\varepsilon,\,a+\varepsilon)\big)
= \nu (\{a\}) = \theta\,,
\end{equation*}
there exists $\varepsilon_1 \in \left(0, \frac{1}{3}\right)$ such that
\begin{equation*}
\theta_1 := \nu\big((a-3\varepsilon_{1},a+3\varepsilon_{1})\big)<1\,,
\qquad
p \cdot \nu \big( ( a- 3\varepsilon_{1}, a + 3\varepsilon_{1})\big)
= p\cdot \theta_{1} < 1\,.
\end{equation*}
Put now in Lemma~\sref{lem32},
$Q = [a- \varepsilon_{1}, a + \varepsilon_{1}]$,
$\Omega = [a- 2\varepsilon_{1}, a + 2\varepsilon_{1}]$,
$\rho = \nu $,
$\omega (x) = m (x)$,
$\Psi (x, t) = \varphi (x ; t)^{p}$.
Then the condition \eqref{lem32.1}  is satisfied because
$\omega (Q) = 2 \varepsilon_{1} \in (0, +\infty)$ and
$\rho \left( \Omega\right)
= \nu \left( \Omega\right)
= \nu \left( [a- 2\varepsilon_{1}, a + 2\varepsilon_{1}]\right)
\in \left(0, \nu\big(( a-3\varepsilon_1, a+3\varepsilon_1) \big)\,\right)
= (0,\theta_1) \subset (0,1)$.
When $x \in Q = [a- \varepsilon_{1}, a + \varepsilon_{1}] $ and
$t \in \Omega = [a- 2\varepsilon_{1}, a + 2\varepsilon_{1}]$
we have $|x-t| \leq 3 \varepsilon_{1} < 1$  and therefore
$\Psi (x , t ) \geq 1$ which gives the validity of \eqref{lem32.2}(a)
and the inequality
$\Psi(x,t)^{\nu\left( [a-2\varepsilon_{1},a+2\varepsilon_{1}]\right)}
\leq \Psi(x,t)^{\theta_{1}}$.
Thus, the integral in \eqref{lem32.2}(b) can be estimated as
\begin{equation*}
\int\limits_{Q} \Psi(x,t)^{\rho (\Omega)}\,dx
\leq \int\limits_{a-\varepsilon_{1}}^{a+\varepsilon_{1}} \Psi(x,t)^{\theta_{1}}\,dx
= \int\limits_{a-\varepsilon_{1}}^{a + \varepsilon_{1}}
  \left(\frac{\sqrt{1+t^2}}{|\,t-x\,|}\right)^{p\cdot\theta_{1}} dx
= \left(1 + t^2\right)^{p \cdot  \theta_{1}/2}
  \int\limits_{a- \varepsilon_{1}}^{a + \varepsilon_{1}}
  \frac{1}{|\, t - x \, |^{p \cdot  \theta_{1}}} d x \ ,
\end{equation*}\noindent
where
\begin{equation*}
\int_{a- \varepsilon_{1}}^{a + \varepsilon_{1}}
\frac{1}{|\, t - x \, |^{p \cdot  \theta_{1}}} d x \leq \frac{1 }{1- p \cdot
\theta_{1}} \cdot
\max_{x \in Q \ , \ t \in \Omega }  |\, t - x \, |^{1 - p \cdot  \theta_{1}} =
\frac{(3 \varepsilon_{1})^{1 - p \cdot  \theta_{1}} }{1- p \cdot  \theta_{1}} \ ,
\end{equation*}
and we can set in \eqref{lem32.2}(b)
\begin{equation*}
D_{\Psi} =  \left(1 + (|a| + 2\varepsilon_{1})^2\right)^{p \cdot  \theta_{1}/2}
\cdot \frac{(3 \varepsilon_{1})^{1 - p \cdot  \theta_{1}} }{1- p \cdot  \theta_{1}}
< \infty \ .
\end{equation*}
Applying the result of Lemma~\sref{lem32} we get
\begin{equation}\label{6.10}
\int_{a-\varepsilon_{1}}^{a+\varepsilon_{1}}
\exp\left(
p \cdot \int_{a-2\varepsilon_{1}}^{a+2\varepsilon_{1}}
\log \lfrac{ \sqrt{1+t^2} }{ |\,t-x\,| }\, d\nu(t) \right)\;dx
\leq D_{\Psi}\;.
\end{equation}

For $x \in Q = [a- \varepsilon_{1}, a + \varepsilon_{1}]$ and
$|t-a| \geq 2\varepsilon_{1}$
we obviously have,
$\max\limits_{|t-a| \geq 2\varepsilon_{1}} \varphi (x ; t)
\leq \varepsilon_{1}^{-1} {\sqrt{1 + (|a|+2\varepsilon_{1})^2}}$
from which it follows that
\begin{equation}\label{6.11}
\int_{\R\setminus \Omega} \log \Psi (x , t ) d \nu (t)
\leq  p \cdot \nu (\R) \cdot
\log \frac{\sqrt{1 + (|a|+2\varepsilon_{1})^2}}{\varepsilon_{1}}\;.
\end{equation}
Therefore \eqref{6.10} and \eqref{6.11} yield
\begin{align*}
\int_{a-\varepsilon_{1}}^{a+\varepsilon_{1}}
\alpha_{\nu}(x)^p\,dx
&= \int_{a-\varepsilon_{1}}^{a+ \varepsilon_{1}}
\exp\left( \int_{\R} \log\Psi(x,t)\,d\nu(t) \right)\;d x  \\
&= \int_{a-\varepsilon_{1}}^{a+\varepsilon_{1}}
\exp\left(
\int_{\Omega} \log \Psi(x,t)\,d\nu(t) \right)
\,\cdot\,
\exp\left(
\int_{\R\setminus \Omega} \log \Psi(x,t)\,d\nu(t) \right)\;dx  \\
& \leq
\left( \lfrac{\sqrt{1+(|a|+2\varepsilon_{1})^2}}{\varepsilon_{1}}
\right)^{\rho \cdot p}
\cdot \int_{a-\varepsilon_{1}}^{a+\varepsilon_{1}}
\exp\left( \int_{\Omega} \log \Psi(x,t)\,d\nu(t) \right) \;dx \\
& \leq
\left( \lfrac{\sqrt{1+(|a|+2\varepsilon_{1})^2}}{\varepsilon_{1}}
\right)^{\rho\cdot p}
\cdot \left(1+(|a|+2\varepsilon_{1})^2\right)^{p\cdot\theta_{1}/2}
\cdot \frac{(3\,\varepsilon_{1})^{1-p\cdot\theta_{1}}}{1-p\cdot\theta_{1}} \\
& < +\infty\;,
\end{align*}
what was to be proved.
\end{proof}


\vspace{0.1cm}
\subsubsection{} \
Now we take an arbitrary $T > 1$ and, for convenience,  designate
\begin{equation*}
J(\nu) := \left\{ x\in\R \mid
\nu \left(\left\{x\right\}\right)> 0\;\right\}\,, \quad
\vartheta := \nu \left(\R\right) \in (0,1]\,.
\end{equation*}
The set $J(\nu)$ is at most countable and
\begin{equation*}
\sum_{x \in J(\nu)} \nu \left(\left\{x\right\}\right)
\leq \nu \left(\R\right)
= \vartheta \in (0,1]\,, \qquad
\sum_{x \in {{J}}(\nu)} \nu \left(\left\{x\right\}\right)
+ \nu \left(\R\setminus J(\nu) \right)
= \vartheta\,.
\end{equation*}
Clearly, $J(\nu) \subset {\rm{supp}}\,\nu$,
$\nu\left( {\rm{supp}}\,\nu \right) = \nu \left(\R\right)$
and hence
\begin{equation} \label{6.12}
\sum_{x \in J(\nu)} \nu \left(\left\{x\right\}\right) +
\nu \left( (\R\setminus J(\nu)) \cap {\rm{supp}}\,\nu \right)
= \vartheta\,.
\end{equation}
Since \eqref{2.3} gives a complete description of all functions
in the class $\primP  \cap  \log \mathcal{P}$  whose corresponding
measure $\nu$ by \eqref{cor1.3} has a support composed
of no more than one point, then it follows from
$f \in \left( \primP \cap \log \mathcal{P} \right) \setminus \Delta_{\primP}$
that ${\rm{supp}}\,\nu$ consists of at least two distinct points.
Thus, \eqref{6.12} yields that the sum in the left-hand side of
\eqref{6.12} contains at least two nonzero summand each of which is
strictly less than $\vartheta \leq 1$ and therefore is strictly less
than $1$.
So that
\begin{equation*}
J_{\nu}(\R) < \vartheta \leq 1\,,
\end{equation*}
and for each point of the compact set $[-T,T]$ and for any $p > 1$ with
$p \cdot J_{\nu} ([-T,T]) \leq p \cdot J_{\nu}(\R) < 1$ it is possible
to apply Lemma~\sref{l6.1} to get the following open cover of $[-T,T]$:
\begin{equation*}
[-T,T] \subset
\bigcup_{x\in[-T,T]}\big(x-\varepsilon(x,p),\,x +\varepsilon(x,p)\big)\,.
\end{equation*}
Choosing a finite subcover
\begin{equation*}
[-T,T] \subset \bigcup\limits_{k =1}^{N_T}
\big(\,x_k^{T}-\varepsilon(x_k^{T},p),\,x_k^{T}+\varepsilon(x_k^{T},p)\big)\,,
\quad x_k^{T} \in [-T,T]\,,\quad k=\overline{1,N_T}\,,\quad N_T\in\N\,,
\end{equation*}
we get
\begin{equation*}
\int\limits_{x_k^{T} - \varepsilon (x_k^{T}, p) }^{x_k^{T} + \varepsilon (x_k^{T}, p)}
\alpha_{\nu}(x)^p\,dx < +\infty\,, \quad k=\overline{1,N_T}\,,
\end{equation*}
and thus,
\begin{equation}\label{6.14}
\int\limits_{-T}^{T} \alpha_{\nu}(x)^p\,dx < +\infty\,, \quad
1 < p < \frac{1}{J_{\nu}(\R)}\,, \quad T>0\,.
\end{equation}

\vspace{0.1cm}
\subsubsection{} \
The properties \eqref{a1} imply that \eqref{th3.2} follows from the
validity of the right-hand side inequality in  \eqref{th3.2}.
Therefore to finish the proof of Theorem~\ref{th3} it  suffices to prove
(see \eqref{cor1.6}) that  both of the following inequalities hold:
\begin{equation} \label{6.15}
\int_{-\infty}^{0}
\exp\left(
p \cdot \int_{-\infty}^{+\infty} \log \varphi(x;t)\,d\nu(t) \right)
\;dx\,,\quad
\int_{0}^{+\infty}
\exp\left(
p \cdot \int_{-\infty}^{+\infty} \log \varphi(x;t)\,d\nu(t) \right)
\;dx < +\infty\,.
\end{equation}
We introduce now the function
\begin{equation*}
\widehat{\nu} (t)
:= \nu(+\infty) + \nu(-\infty) - \nu(-t)\,,\quad t\in\R\,.
\end{equation*}
It is evident that
$\widehat{\nu}\in M^{\uparrow }(\R)$,
$\widehat{\nu}(-\infty) = {\nu}(-\infty)$,
$\widehat{\nu}(+\infty) = {\nu}(+\infty)$,
and for any $b > 0$,
$\widehat{\nu}(b-0)  = {\nu}(+\infty) + {\nu}(-\infty) - {\nu}(-b+0)$,
$\widehat{\nu}(-b+0) = {\nu}(+\infty) + {\nu}(-\infty) - {\nu}( b-0)$.
Thus, $\supp \widehat{\nu} = - \supp \nu$
and
\begin{equation} \label{6.18}
\widehat{\nu}(\R) = \nu(\R)\,,\quad
J_{\widehat{\nu}}(\R) = J_{\nu}(\R)\,,\quad
\widehat{\nu}\big( (-b,b)\big) = \nu \big( (-b, b) \big)\,,\quad
b > 0\,.
\end{equation}
It follows from $\varphi (- x ; t) \equiv  \varphi (x ; -  t)$, $x, t \in \R$,
that for arbitrary $0< b \leq +\infty$ we have
\begin{align*}
\int\limits_{-b}^{0}
\exp\left(
p \cdot \int_{-\infty}^{+\infty} \log\varphi(x;t)\,d\nu(t) \right)\;dx
&= \int\limits_{0}^{b}
\exp\left(
p \cdot \int_{-\infty}^{+\infty} \log\varphi(-x;t)\,d\nu(t) \right)\;dx \\
&= \int\limits_{0}^{b}
\exp\left(
p \cdot \int_{-\infty}^{+\infty} \log\varphi(x;-t)\,d\nu(t) \right)\;dx \\
&= \int\limits_{0}^{b}
\exp\left(
p \cdot \int_{-\infty}^{+\infty} \log\varphi(x;t)\,d\widehat{\nu}(t)
\right)\;dx\,,
\end{align*}
and so, \eqref{6.14} and \eqref{6.15} can be rewritten as follows,
\begin{equation}\label{6.24}
\int\limits_{0}^{T}
\left(\alpha_{\nu} (x)^p + \alpha_{\widehat{\nu}}(x)^p\right)\,dx < +\infty\,,
\quad 1 < p < \dfrac{1}{J_{\nu}(\R)}\,, \quad T > 0\,;
\end{equation}
\begin{equation}\label{6.19}
\int_{0}^{+\infty}
\exp\left(
p \cdot \int_{-\infty}^{+\infty} \log\varphi(x;t)\,d\sigma(t) \right)\,dx
< +\infty\,, \quad \sigma \in \left\{\nu,\widehat{\nu}\right\}\,,
\end{equation}
where
\begin{equation*}
\alpha_{\sigma} (x) :=
\exp\left( \int_{-\infty}^{+\infty} \log\varphi(x;t)\,d\sigma(t) \right)\,
\quad x \geq 0\,, \quad \sigma \in \left\{\nu ,\widehat{\nu} \right\}\,.
\end{equation*}
Fix an arbitrary $p\in(\,1/\nu(\R),\,J_{\nu}(\R)\,)$.
In view of \eqref{6.18},
$p \cdot \nu \big( \R\big) \equiv p \cdot \widehat{\nu}\big(\R\big)> 1$
and since $ \nu(\R) = \lim_{y\to +\infty} \nu\big( (-y,y) \big)$,
it is possible to find $A\geq1$ such that
$p\cdot\nu\big((-A,A)\big)\equiv p\cdot\widehat{\nu}\big((-A,A)\big >1$
and $\nu \big((-A,A)\big)\equiv\widehat{\nu}\big((-A, A)\big)
\geq \nu \big( \R\big) - {1}/({2\cdot p}) \equiv
\widehat{\nu} \big( \R\big) - {1}/({2\cdot p})$, i.e.,
\begin{equation}\label{6.25}
A \geq 1\,,\quad
p \cdot \sigma \big( (-A,A) \big) > 1\,, \quad
\sigma \big( (-A,A) \big) \geq \sigma\big(\R\big) - \frac{1}{2\cdot p}\,,
\quad \sigma \in \left\{\nu,\widehat{\nu} \right\}\,.
\end{equation}

Applying \eqref{6.24} with $T = 4 A$ we obtain
\begin{equation*}
\int\limits_{0}^{4A} \alpha_{\sigma}(x)^{p}\,dx < +\infty\,, \quad
\sigma \in \left\{\nu , \widehat{\nu} \right\}\,,
\end{equation*}
and to establish the validity of  \eqref{6.19} it remains to prove that
\begin{equation}\label{6.21}
I_{A, \sigma} := \int\limits_{4 A}^{+\infty}
\exp\left(
p \cdot \int_{-\infty}^{+\infty} \log\varphi(x;t)\,d\sigma(t) \right)\,dx
< +\infty\,,\quad \sigma \in \left\{\nu, \widehat{\nu}\right\}\,.
\end{equation}

\vspace{0.1cm}
\subsubsection{} \
For $\sigma \in \left\{\nu , \widehat{\nu} \right\}$ and $x \geq 4 A$
we have
\begin{equation*}
\max_{|t| \leq A} \varphi (x ; t) = \dfrac{\sqrt{1+A^2}}{|x - A|} \leq 1\,,
\end{equation*}
from which we get
\begin{equation*}
\int_{(-A,A)} \log\varphi(x;t)\,d\sigma(t)
\leq \int_{(-A,A)} \log\dfrac{\sqrt{1+A^2}}{|x-A|}\,d\sigma(t)
= \sigma \big((-A , A)\big) \log \dfrac{\sqrt{1+A^2}}{|x-A|}\,,
\end{equation*}
and
\begin{align}
I_{A, \sigma}
& = \int\limits_{4 A}^{+\infty}
\exp\left(
p \cdot \int_{(-A,A)} \log\varphi(x;t)\,d\sigma(t) +
p \cdot \int_{\R\setminus (-A,A)} \log\varphi(x;t)\,d\sigma(t) \right)\,dx
\label{6.22} \\
&\leq \int\limits_{4 A}^{+\infty}
\left(\dfrac{\sqrt{1+A^2}}{|x-A|}\right)^{p\cdot\sigma\big((-A,A)\big)}
\exp\left(
p \cdot \int_{\R\setminus (-A,A)} \log\varphi(x;t)\,d\sigma(t) \right)\,dx
\nonumber \\
& = \sum_{n\geq 0} I_{A,\sigma}\left(4A \cdot 2^n\right)\,, \nonumber
\end{align}
where
\begin{equation*}
I_{A, \sigma} (a)
:=  \int\limits_{ a}^{2 a}
\left(\dfrac{\sqrt{1+A^2}}{|x-A|}\right)^{p\cdot\sigma((-A,A))}
\exp\left(
p \cdot \int_{\R\setminus (-A,A)}\log\varphi(x;t)\,d\sigma(t) \right)\,dx\,,
\end{equation*}
and  $a \in \left\{4 A \cdot 2^n\right\}_{n \geq 0 }\ $.

It can be easily verified that
$ a \geq 4 A \geq 4$ implies $ \varphi (x ; t) \leq \sqrt{5}$
for all  $x \in [a, 2a]$, $t \notin [\frac{a}{2}, 4 a ]$ and
\begin{equation*}
\dfrac{\sqrt{1+A^2}}{|x - A|} \leq  \dfrac{\sqrt{1+A^2}}{|a - A|}
\leq \dfrac{2 A }{a}\,,\quad x \in [a, 2a]\,.
\end{equation*}
Using the inequalities \eqref{6.25}, we get
\begin{align*}
p \cdot \int_{\R\setminus (-A,A)} \log \varphi (x ; t)\,d\sigma (t)
& \leq p \cdot \sigma \big( \R \setminus (-A, A) \big) \cdot \log \sqrt{5}
  + p \cdot \int_{\left[\tfrac{a}{2}\, , \ 4 a\right]}
  \log \varphi(x;t)\,d\sigma(t) \\
& \leq \frac{1}{4} \log 5 +  p \cdot \int_{\left[\tfrac{a}{2}\, , \ 4 a\right]}
  \log \varphi(x;t)\,d\sigma(t)\,,
\end{align*}
and so,
\begin{equation}\label{6.23}
I_{A, \sigma} (a) \leq 2 \left(\dfrac{2A}{a}\right)^{p\cdot\sigma\big((-A,A)\big)}
\cdot I_{\sigma} (a)\,, \quad
I_{\sigma} (a) := \int\limits_{ a}^{2 a}
\exp\left(
p \cdot \int_{\left[\tfrac{a}{2}\,,\,4a\right]} \log\varphi(x;t)\,d\sigma (t)\right)
\,dx\,.
\end{equation}
To estimate $I_{\sigma} (a)$ we put  in Lemma~\sref{lem32}, $Q = [a, 2a ]$,
$\Omega = [a/2, 4a ]$, $\rho =  \sigma $,
$\omega (x) = m (x)$, $\Psi (x, t) = \varphi (x ; t)^{p}$.
Observe that  \eqref{6.25} and ${a} \geq 4 A \geq 4$ yield
$\sigma \big([\tfrac{a}{2}, 4 a]\big)
 \leq \sigma \big([ 2 A, +\infty )\big)
 \leq \sigma \big([  A, +\infty )\big)
 \leq \sigma (\R \setminus (-A, A))
 = \sigma (\R) - \sigma ( (-A, A))
 \leq 1/(2\cdot p)$.
If $\sigma \big([\tfrac{a}{2}, 4 a]\big) = 0$, then obviously
\begin{equation}\label{6.26}
I_{\sigma} (a) \leq a\,.
\end{equation}
Otherwise,  the condition \eqref{lem32.1}  is satisfied because
$\omega (Q) = a \in [4, +\infty)$ and
$\rho \left( \Omega\right)
 = \sigma \left( \Omega\right)
 = \sigma \left( [a/2, 4a ]\right) \in  (0, 1/2p]
 \subset (0,1)$.
When $ x \in Q = [a, 2a ]$ and $t \in \Omega = [a/2, 4a ]$ we have
$|x-t| \leq 3 a$  and therefore
$\varphi (x ; t) \geq ({3a})^{-1} {\sqrt{1+ a^{2}/4}}\geq 1/6$
which gives \eqref{lem32.2}(a) while the integral in \eqref{lem32.2}(b)
can be estimated as
\begin{align*}
\int_{Q} \Psi (x , t )^{\rho (\Omega)} d x
&= \int\limits_{ a}^{2 a} \
   \left(\frac{\sqrt{1+t^2}}{|t - x|}\right)^{{\fo{p\,\sigma\big([{a}/{2}\,,\,4 a]\big)}}}
   \,dx  \\
&\leq \left(\sqrt{1+16 a^2}\right)^{{\fo{p\,\sigma\big([{a}/{2}\,,\,4 a]\big)}}}
   \cdot \int\limits_{ a}^{2 a}
   \dfrac{1}{|t-x|^{{\fo{p\,\sigma\big([{a}/{2}\,,\,4 a]\big)}}}}\;dx \\
& \leq \left( 5a \right)^{{\fo{p\,\sigma\big([{a}/{2}\,,\,4 a]\big)}}}
   \cdot 2 \cdot
   \dfrac{(3a)^{{\fo{1 - p\,\sigma\big([{a}/{2}\,,\,4a]\big)}}}}
         {1-p\sigma \big([{a}/{2}\,,\,4 a]\big)} \\
& \leq \dfrac{10 a}{ 1 - p \sigma \big([{a}/{2}\,,\,4 a]\big)}
  \leq 20 a\,.
\end{align*}
Setting in \eqref{lem32.2}(b) $D_{\Psi} = 20 a$ we can apply the result
of Lemma~\sref{lem32} to get, in view of \eqref{6.26}, $ I_{\sigma} (a)
\leq 20 a$ and then by \eqref{6.25}  and \eqref{6.23}
\begin{equation*}
I_{A, \sigma} (a)
\leq 40 a \left(\dfrac{2 A }{a} \right)^{{\fo{p \cdot \sigma \big((-A , A)\big)}}}
\leq \dfrac{40\cdot (2 A)^p}{a^{{\fo{p \cdot \sigma \big((-A , A)\big) - 1}}}}
\leq \dfrac{40\cdot (2 A)^p}{\sqrt{a}}\,.
\end{equation*}
Now it follows from \eqref{6.22} that
\begin{equation*}
I_{A, \sigma}
\leq \sum_{n \geq 0 } I_{A, \sigma} \left(4A\cdot 2^n\right)
\leq 20 \cdot (2 A)^p \cdot \sum_{n \geq 0} \frac{1}{ 2^{\hspace{0.05cm}{\fo{n/2}}}}
\leq 80 \cdot (2 A)^p < +\infty\,,
\end{equation*}
which gives  \eqref{6.21} and finishes the proof of \eqref{th3.2}.

To prove \eqref{th3.3} assume that
$\exp({- p \cdot U^{f}(x)}) \in L_{1} (\R, d x/(1+x^2))$ for some
${1}/{\nu (\R)} < p < {1}/{J}_{\,\nu} (\R)$.
By setting   $y = \exp (- p \cdot U^{f}(x))$ in \eqref{3.1.1} we get
from it and \eqref{th3.2} that for arbitrary $\Lambda > 0$
\begin{align*}
4 \log (1 + \Lambda)
& = \int_{-\Lambda}^{\Lambda} \frac{2}{1+|x|} \, d x \\
& \leq \int_{-\Lambda}^{\Lambda} e^{\hspace{0.05cm}{\fo{p \cdot U^{f}(x)}}}\,dx +
       \int_{-\Lambda}^{\Lambda} \frac{e^{\hspace{0.05cm}{\fo{-p\cdot U^{f}(x)}}}}{(1+|x|)^{2}}\,d x \\
& \leq \int_{\R} e^{\hspace{0.05cm} {\fo{p \cdot U^{f}(x)}}}\,d x +
       \int_{-\Lambda}^{\Lambda} \frac{e^{\hspace{0.05cm}{\fo{- p\cdot U^{f}(x)}}}}{(1+|x|)^{2}}\,dx\,,
\end{align*}
which leads to a contradiction as $\Lambda \to +\infty$.
Thus,
$\exp({- p \cdot U^{f}(x)}) \notin L_{1} (\R, d x/(1+x^2))
\supset L_{q} (\R, d x/(1+x^2))$
for all $q \geq 1$ (see \cite[p.110]{bar}).
Together with inequalities \eqref{a1}(b) this proves \eqref{th3.3}
and completes the proof of Theorem~\sref{th3}.

\vspace{0.35cm}
\subsection{Proof of Theorem~\ref{pr1}}$\phantom{a}$ \\

\vspace{-0.2cm}
Using  Theorem~\ref{cor1} for $f:= \log \varphi$ we get
(see \cite[p.27]{D}) that $f$  is a non-constant function in the class
$\log \mathcal{P}$, $\varphi = e^{f}$ and
\begin{equation}\label{7.1}\hspace{-0.4cm}
\left|\varphi(x,y)\right|\!=\! e^{\hspace{0.05cm}{\fo{U^{f}(x,y)}}}\,,\
U^{\varphi}(x, y)\! =\! e^{\hspace{0.05cm}{\fo{U^{f}(x, y)}}} \cos V^{f}(x,y)\,,\
V^{\varphi}(x, y)\! =\! e^{\hspace{0.05cm}{\fo{U^{f}(x, y)}}} \sin V^{f}(x,y)\,.\hspace{-0.3cm}
\end{equation}
Furthermore, since $\varphi \in \mathcal{P}$ is a non-constant function
(see \cite[p.18]{D}),
\begin{equation}\label{7.2}
\arg \varphi (x + i y) = V^{f} (x,y) \in (0, \pi)\,.
\end{equation}
Now it follows directly from \eqref{a1} that each of the properties
(1) and (2) in Theorem~\ref{pr1}  is equivalent for the function $f$
to be in the class  $  {\mathcal{P}}_{\tn{\int} }  \cap  \log {\mathcal{P}}  $.

Since
\begin{align*}
\begin{vmatrix}
U^{\varphi}(x_1,y) & U^{\varphi}(x_2,y) \\
V^{\varphi}(x_1,y) & V^{\varphi}(x_2,y)
\end{vmatrix}
&= e^{{\fo{U^{\!f}(x_1,y)+U^{\!f}(x_2,y)}}}\cos V^{f}(x_1,y)\sin V^{f}(x_2,y )  \\
&\qquad -e^{{\fo{U^{\!f}(x_2,y)+U^{\!f}(x_1,y)}}}\cos V^{f}(x_2,y)\sin V^{f}(x_1,y) \\
&= e^{{\fo{U^{\!f}(x_1,y)+U^{\!f}(x_2,y)}}} \sin\big(V^{f}(x_2,y)-V^{f}(x_1,y)\big)
\end{align*}
for $y > 0$ and $-\infty < x_1  < x_2 +\infty$, it follows from \eqref{7.2}
that \eqref{a1}(a) is equivalent to the condition (3) of Theorem~\ref{pr1}.

Similarly, \eqref{7.2} and identity
$$
\frac{\partial }{\partial y} e^{{\fo{2 U^{\!f}(x, y)}}}
= \frac{\partial }{\partial y}
  \left( U^{\varphi}(x,y)^{2} + V^{\varphi}(x,y)^{2} \right)
= 2 \left(U^{\varphi}(x,y) \cdot U^{\varphi}_{y}(x,y)
  + V^{\varphi}(x,y) \cdot V^{\varphi}_{y}(x,y)\right)
$$
imply the equivalence of \eqref{a1}(b) and the condition (4) of
Theorem~\ref{pr1}. The proof is complete.


\vspace{0.25cm}
\subsection{Proofs of Theorem~\ref{cor3} and  Corollary~\ref{cor4}}

\subsubsection{Proof of Theorem~\ref{cor3}.}\las{91}

Observe that the first item of Theorem~\ref{cor3} follows directly from
Theorem~\sref{cor1} and \eqref{9}.
If $\varphi \in \mathcal{P}_{\log} \setminus \exp{\Delta_{\primP}}$
then $\varphi = e^{f}$, where
$f\in\left(\primP\cap\log\mathcal{P}\right) \setminus \Delta_{\primP}$.
Using the notation of Theorems~\ref{cor1} and~\ref{pr1} we can apply  the
results of Theorem~\ref{th3} to the function $f$ in order to get with the
help of the relations \eqref{7.1} and \eqref{7.2} the validity of all
statements of Theorem~\ref{cor3} except \eqref{cor3.8}.
Equality   \eqref{cor3.8} holds true due to the well-known consequence of
M.Riesz's theorem \cite[p.128]{ko}.

\subsubsection{Proof of Corollary~\ref{cor4}.}\las{93}

If $\varphi (z) = e^{f}$, where
$ f = \int_{0}^{1} \log \frac{1}{1-tz}\,d\mu(t)$, then by
Corollary~\ref{cor2} $f (0) = 0$, $f \in \hol((-\infty,1))$ and
$f \in \primP  \cap  \log \mathcal{P} $, which imply the validity of
all statements of Corollary~\ref{cor4}.

Conversely, assume that $\varphi(0)=1$, $\varphi\in\hol((-\infty,1))$
and  $\varphi\in \mathcal{P}_{\log}$.
Then there exists $ f \in \primP  \cap  \log \mathcal{P} $ such that
$\varphi (z) = e^{f}$ and $f (0) = 0$.
Since $\varphi \in  \mathcal{P}$ and $\varphi \in \hol((-\infty,1))$
the property \eqref{a2} and representation \eqref{2} imply that
$\IM \varphi (x+i0) = 0 $ for any $x < 1$.
Then it follows from \eqref{cor3.7} that the function $\nu$
corresponding to $\varphi $ by the exponential representation
\eqref{cor3.1} satisfies $\sin \pi \nu (x) = 0$ for almost all
$x < 1$.
Since by Theorem~\ref{cor3} $\nu$ is a nondecreasing
function on $\R$ with values in $[0,1]$ then  for almost all $x<1$
either $\nu (x) = 0$ or $\nu (x) = 1$ or $\nu (x) = \chi_{[a, +\infty)}(x)$
for some $a < 1$.

Assume that $\nu (x) = 1$ for almost all $x<1$.
Then $\nu (x) \equiv 1$ on $\R$ and in view of \eqref{cor3.1}, $\varphi(z)$
is a constant function and so $f$ is constant as well.
By Corollary~\sref{cor2} $f \equiv 0$ and $f$ can be represented in
the form \eqref{cor2.1}, from which \eqref{cor4.1} follows.

If for almost all $x<1$, $\nu (x) = \chi_{[a, +\infty)}(x)$ for some $a < 1$
then $\nu (x) \equiv \chi_{[a, +\infty)}(x)$ on $\R\setminus \{a\}$ and
it follows from \eqref{cor3.1} and $\varphi (0) = 1 $ that
$\varphi(z) = a / (a - z)$, $z \in \H$. By \eqref{2aa} this means
that $\varphi(z) = a / (a - z)$, $z \in \C$, which contradicts
$\varphi \in \hol((-\infty,1))$.

Finally, $\nu (x) = 0$ for almost all $x<1$ implies $\nu (x) = 0$ for all
$x<1$.
Due to the formulas \eqref{5.1} and \eqref{5.2}, the function $f$ can be
represented in the form \eqref{cor2.1} with $\mu$ given by \eqref{5.3},
which means the validity of the representation \eqref{cor4.1}.
The proof of  Corollary~\ref{cor4} is complete.


\vspace{0.25cm}
\subsection{Proof of Theorem~\ref{th4}}

\subsubsection{Necessity.}

Applying the results of Theorem~\ref{cor3}, 2) to
$\varphi \in \mathcal{P}_{\log} \setminus \exp{\Delta_{\primP}}$
we get the existence of $1< p < +\infty$  such that
$\varphi \in \mathcal{H}_p \, (\H)$,
$\varphi (x + i 0) = \widetilde{v}(x) + i v (x)$ for almost all
$x \in \R$ and $v , \widetilde{v} \in L_{p}(\R, d x)$.
Furthermore, the  Schwarz integral formula applied to $\varphi (z) / i$
gives the validity of \eqref{th4.2} (see \cite[p.145]{kres}) and since
$\varphi \in \mathcal{P}$ the function $v (x)$ is non-negative everywhere
on $\R$.
Taking the limit when $y \downarrow 0$ in the inequalities (3) of
Theorem~\ref{pr1} we get \eqref{th4.1} and thus finish the proof of
 necessity of Theorem~\sref{th4}.

\subsubsection{Sufficiency.}

Let $\varphi$ be defined in \eqref{th4.2}.
Since $v \in L_p (\R, d x)$ is  non-negative and nonzero
it follows from  \cite[p.128]{ko} that
$\varphi \in \mathcal{P} \cap \mathcal{H}_p \, (\H)$ and for  all
$ x \in \mathcal{D} (\varphi)$ we have
$\varphi (x+i0) = \widetilde{v} (x) + i v (x)$ where
$V^{\varphi}(x) = v (x)$, $U^{\varphi}(x) = \widetilde{v} (x)$ and
$ \widetilde{v} \in L_p (\R, d x)$.

Introduce the function
$h := \log \varphi \in \log \mathcal{P} \subset \mathcal{P}$.
According to \eqref{3} there exist a real number $\beta$ and
$\rho \in L_{\infty}(\R)$, satisfying $0 \leq \rho (x) \leq 1$
for almost all $x \in \R$, such that
\begin{equation}\label{9.2}
h(z) = \beta  +
\int_{-\infty}^{+\infty} \left(\frac{1}{t - z} - \frac{t}{1 + t^2}\right)\,
\rho(t)\,d t\,,\quad z \in \H\,.
\end{equation}
Therefore
\begin{equation*}
V^{h} (x,y) =
\int_{-\infty}^{+\infty} \frac{y}{(x-t)^{2}+y^{2}}\,\rho(t)\,dt\,,\quad
x\in\R\,, \quad y > 0\,,
\end{equation*}
and in view of \cite[p.107]{ko} there exists a Borel subset
$\mathcal{D}_1 (h)$ of the real line such that
\begin{equation*}
V^{h}(x) = \pi\cdot\rho(x)\,,\quad x\in\mathcal{D}_1(h)\subset\mathcal{D}(h)\,,
\quad m\left( \R\setminus\mathcal{D}_1(h)\right) = 0\,.
\end{equation*}
For each $x \in \mathcal{D} (\varphi) \cap \mathcal{D}_{1} (h)$
we can take the limit $y \downarrow 0$ in the equation
$U^{\varphi}(x,y)+iV^{\varphi}(x,y)=\exp\left( U^{h}(x,y)+iV^{h}(x,y)\right)$
to get
\begin{equation}\label{9.3}
\widetilde{v}(x) = e^{\hspace{0.05cm} {\fo{U^{h} (x)}}} \cos \pi \rho(x)\,,\qquad
v(x) = e^{ \hspace{0.05cm}{\fo{U^{h}(x)}}} \sin \pi \rho(x)\,,\qquad
x \in \mathcal{D} (\varphi) \cap \mathcal{D}_{1}(h)\,.
\end{equation}
Substituting  \eqref{9.3} into the determinant from \eqref{th4.1} gives
\begin{align*}
0
&\leq
  \begin{vmatrix}
  \widetilde{v} (x_{1})  & \widetilde{v} (x_{2}) \\
  v_{}(x_1) & v_{}(x_2) \\
  \end{vmatrix}
= \begin{vmatrix}
  e^{\hspace{0.05cm}{\fo{U^{h}(x_{1})}}} \cos\pi\rho(x_{1}) & e^{U^{h}(x_{2})} \cos\pi\rho(x_{2}) \\
  e^{\hspace{0.05cm}{\fo{U^{h}(x_{1})}}} \sin\pi\rho(x_{1}) & e^{U^{h}(x_{2})} \sin\pi\rho(x_{2}) \\
  \end{vmatrix}  \\
&= e^{\hspace{0.05cm}{\fo{U^{h}(x_{1}) + U^{h}(x_{2})}}}
   \left( \cos\pi\rho(x_{1}) \sin\pi\rho(x_{2})
   - \sin\pi\rho(x_{1}) \cos\pi\rho(x_{2}) \right) \\
&= e^{\hspace{0.05cm}{\fo{U^{h} (x_{1}) + U^{h} (x_{2})}}} \sin\pi\left(\rho(x_{2}) - \rho(x_{1})\right)\,,
\end{align*}
from which we deduce that $\rho(x_{1}) \leq \rho(x_{2}) $ for  all
$x_1<x_2$, $x_1,x_2\in E:=\mathcal{D}(\varphi)\cap\mathcal{D}_{1}(h)$.
Then the function
$\widehat{\rho} \in  M^{\uparrow } (\R)$  defined as
$\widehat{\rho} (x) := \lim_{\, t \to  x \, , \  t < x \, , \ t \in E}\rho (t)$,
$x \in \R$, is equal to $\rho$ almost everywhere on $\R$ (see \eqref{aux4},
\eqref{aux12} ) and is equivalent to $\rho$ in the space  $L_{\infty}(\R)$.
By Theorem~\ref{cor1} and \eqref{9.2} this means that
$h\in \primP \cap \log\mathcal{P}$ and therefore
$\varphi\in\mathcal{P}_{\log}$.

We now prove that $\varphi\notin\exp{\Delta_{\primP}}$.
Consider an arbitrary function in the class $\exp{\Delta_{\primP}}$
(see \eqref{2.4})
\begin{equation*}
\varphi_{\theta_1,\,a,\,\theta}(z)
= e^{\hspace{0.05cm}{\fo{\beta + i\pi\theta_1\cdot(1-\theta)}}}\cdot\dfrac{1}{(a-z)^{\theta}}\,,
\end{equation*}
where $ \theta,\theta_1\in [0,1]$ and $a,\beta\in\R$.
Equalities \eqref{4.5}, \eqref{4.1} and a straightforward calculation gives
\begin{equation*}
v_{\theta_1,\,a,\,\theta}(x)
:= \lim_{y \downarrow 0} \IM \varphi_{\theta_1,\,a,\,\theta}(x+iy)
= \begin{cases}
    e^{\hspace{0.05cm}{\fo{\beta}}} \cdot
    \dfrac{\sin\left(\pi\theta_1\cdot(1-\theta) + \theta\pi\right)}
          {|a-x|^{\theta}}\,, & \quad a<x\,, \\[0.4cm]
    e^{\hspace{0.05cm}{\fo{\beta}}} \cdot
    \dfrac{\sin\left(\pi\theta_1\cdot(1-\theta)\right)}
          {|a-x|^{\theta}}\,, & \quad a>x\,.
  \end{cases}
\end{equation*}

\vspace{0.25cm} \noindent
It is easy to verify that for any values of $\theta,\theta_1\in[0,1]$,
and $a,\beta\in\R$, either $v_{\theta_1,\,a,\,\theta}(x)=0$ almost everywhere
on $\R$, or $v_{\theta_1,\,a,\,\theta}(x)\notin\bigcup_{p>1} L_p(\R,dx)$.
Thus, $\varphi \in \mathcal{P}_{\log} \setminus \exp{\Delta_{\primP}}$.

Since \eqref{th4.2} is a partial case of the unique representation
\eqref{2} we get the uniqueness of \eqref{th4.2} and complete the
proof of Theorem~\sref{th4}.

\vspace{0.25cm}
\subsection{Proof of Theorem~\ref{th5}}$\phantom{a}$ \\

\vspace{-0.35cm}

Observe first that if $v$ satisfies \eqref{th5.2a} inequality \eqref{th5.2c}
holds for arbitrary $-\infty < x_1 < x_2 < 1$  and
$-\infty < x_1 <  1 <x_2 < +\infty$.
Therefore to check the validity of  \eqref{th4.1} for $v$ satisfying
\eqref{th5.2a}, it  suffices to prove it only for all
$1 < x _1 < x _2 < +\infty $.

\subsubsection{Necessity.}

Let $\Psi$ be universally starlike and $\varphi (z) := \Psi (z)/z$.
Then by \eqref{1} $\varphi (z)$ can be represented in the form
\eqref{cor4.1} where the measure $\mu$  corresponds to $\varphi$.
According to Corollary~\ref{cor4}, $\varphi \in \mathcal{P}_{\log}$.
If $\varphi \in \exp{\Delta_{\primP}}$, then by virtue of \eqref{2.6}
$\varphi$ belongs to the set in \eqref{th5.1}.
Otherwise, $\varphi\in \mathcal{P}_{\log}\setminus \exp{\Delta_{\primP}}$.
In this case, due to Theorem~\ref{th4}, we can find $p>1$ and
$v \in L_p (\R, d x) $ such that the representation \eqref{th5.3} holds,
$v (x) \geq 0$ for all $x\in \R$, $v$ satisfies \eqref{th5.2c} and in view
of \eqref{th4.3}, $\varphi (x+i0) = \widetilde{v}(x)+iv(x)$
$m$-a.e. on $\R$.
Corollary~\ref{cor4} states also that $\varphi \in \hol((-\infty,1))$
and $\varphi(0) =1$.
Since $\varphi \in  \mathcal{P}$ we can apply \eqref{a2} and \eqref{2} to
$\varphi\in\hol((-\infty,1))$ and get $v(x) = \IM \varphi (x+i0) = 0 $
for any $x < 1$.
This finishes the proof of \eqref{th5.2a}.
Finally, \eqref{th5.2b} follows from \eqref{th5.3} and $\varphi(0) =1$.

\subsubsection{Sufficiency.}

Let $\Psi$ belong to the set \eqref{th5.4}.
Then it follows from \eqref{2.6} that $\Psi (z)/z$ can be represented in
the form of \eqref{cor4.1} and therefore by \eqref{1} $\Psi$ is
universally starlike.

Assume now that $\Psi$ is represented in the form \eqref{th5.5}.
Then by Theorem~\ref{th4} $\Psi (z)/z  \in \mathcal{P}_{\log}
\setminus \exp{\Delta_{\primP}}$ and it follows directly
from \eqref{th5.5} and \eqref{th5.2b} that
$\Psi (z)/z \in \hol((-\infty,1))$ and $\Psi^{\, \prime} (0) = 1 $,
respectively.
By Corollary~\ref{cor4} this means that  $\Psi (z)/z$ can be represented
in the form of \eqref{cor4.1} and therefore by \eqref{1} $\Psi$ is
universally starlike.
The proof of Theorem~\ref{th5} is complete.

\vspace{0.5cm}

\end{document}